\newif \ifIMA

\ifIMA
\documentclass{imanum}
\else

\documentclass[10pt]{article}
\usepackage[a4paper,left=2.5cm,right=2.5cm,top=2.5cm,bottom=2.5cm]{geometry}
\fi

\usepackage{amsmath,amsthm,amssymb,bm,xspace,enumerate}
\usepackage[usenames]{xcolor}
\usepackage{graphicx,verbatim,mathtools}
\usepackage[only,llbracket,rrbracket,llparenthesis,rrparenthesis]{stmaryrd}
\usepackage[textsize=scriptsize,color=blue!10!white]{todonotes}
\usepackage[breaklinks,bookmarks=false]{hyperref}
\hypersetup{colorlinks, linkcolor=blue, citecolor=blue,
urlcolor=blue, plainpages=false, pdfwindowui=false,
pdfstartview={FitH}}

\numberwithin{equation}{section}

\ifIMA

\else

\newtheorem{theorem}{Theorem}[section]{\bfseries}{\it}
\newtheorem{proposition}[theorem]{Proposition}{\bfseries}{\it}
\newtheorem{lemma}[theorem]{Lemma}{\bfseries}{\it}
{\bfseries}{\it}
\newtheorem{definition}[theorem]{Definition}{\bfseries}{\it}
{\bfseries}{\it}
{\bfseries}{\it}
\theoremstyle{definition}
\newtheorem{remark}[theorem]{Remark}{\bfseries}{\rmfamily}

\fi

\newcommand{\GD}{\Gamma_{\mathrm D}}
\newcommand{\GN}{\Gamma_{\mathrm N}}

\newcommand{\bH}{{\bm H}}
\newcommand{\bL}{{\bm L}}
\newcommand{\bC}{{\bm C}}

\def\Hunz{{{H_0^1(\Omega)}}}
\def\HdivK{{{\bH(\textup{div};K)}}}

\def\Hdivzero{{{\bH_{0,{\Gamma_{\mathrm{N}}}}(\textup{div},\Omega)}}}

\newcommand{\RTNK}{\bm{RTN}_p(K)}
\newcommand{\RTNKp}{\bm{RTN}_p(K^{'\!\!})}
\newcommand{\RTNKm}{\bm{RTN}_{p-1}(K)}

\newcommand{\Eglob}{E_{\calT\kern-0.2ex,p}}
\newcommand{\Eloc}{e_{K,p}}
\newcommand{\Eloct}{e_{K^{'\!\!},p}}
\newcommand{\Elocp}{e_{K,p-1}}

\newcommand{\p}{\partial}
\newcommand{\ver}{{\bm{a}}}
\newcommand{\Ta}{\calT_{\ver}}
\newcommand{\TF}{\calT_{F}}

\newcommand{\calFaint}{\mathcal{F}_{\ver}^{\mathrm{in}}}
\newcommand{\oma}{{\om_{\ver}}}
\newcommand{\Hdiva}{{{\bH(\textup{div};\oma)}}}
\newcommand{\psia}{\psi_{\ver}}

\newcommand{\saa}{\bm{\sigma}_{\ver}}

\newcommand{\bsig}{{\boldsymbol \sigma}}
\newcommand{\br}{{\bm r}}

\newcommand{\bvh}{\bm{v}_{\calT}}
\newcommand{\bva}{\bm{v}_{\ver}}
\newcommand{\bvK}{\bm{v}_{K}}
\newcommand{\bv}{\bm{v}}
\newcommand{\bn}{\bm{n}}
\newcommand{\bp}{\bm{p}}
\newcommand{\bx}{\bm{x}}

\newcommand{\calA}{\mathcal{A}}
\newcommand{\calF}{\mathcal{F}}
\newcommand{\calT}{\mathcal{T}}

\newcommand{\calV}{\mathcal{V}}
\newcommand{\calP}{\mathcal{P}}

\newcommand{\calVint}{\mathcal{V}_{\Omega}}
\newcommand{\calVext}{\mathcal{V}_{\Gamma}}
\newcommand{\calVextDb}{\mathcal{V}_{\mathrm D}}
\newcommand{\calVextNb}{\mathcal{V}_{\mathrm N}}
\newcommand{\calFK}{\mathcal{F}_{\mathrm{K}}}

\renewcommand{\dim}{d}
\newcommand{\abs}[1]{\lvert#1\rvert}
\newcommand{\tends}{\rightarrow}
\newcommand{\norm}[1]{\lVert#1\rVert}
\newcommand{\Bnorm}[1]{\Bigg\lVert#1\Bigg\rVert}

\DeclareMathOperator{\Curl}{curl}
\DeclareMathOperator{\Div}{div}

\DeclareMathOperator*{\argmin}{arg\,min}

\newcommand{\pair}[2]{\langle#1,#2\rangle}
\newcommand{\mean}[1]{\lbrace\kern-0.85ex\lbrace#1\rbrace\kern-0.85ex\rbrace}
\newcommand{\jump}[1]{\llbracket#1\rrbracket}
\newcommand{\R}{\mathbb{R}}
\newcommand{\om}{\omega}

\newcommand{\VK}{\mathcal{V}_K}
\newcommand{\bLO}{\bm{L}^2(\Omega)}
\newcommand{\bLo}{\bm{L}^2(\om)}

\newcommand{\RTN}{\bm{RTN}_p(\calT)}

\newcommand{\RTNa}{\bm{RTN}_p(\Ta)}

\newcommand{\Hcurl}{\bm{H}(\Curl)}
\newcommand{\Hdiv}{\bm{H}(\Div)}
\newcommand{\Hsdiv}{\bm{H}^s(\Div)}
\newcommand{\HdivO}{\bm{H}(\Div,\Omega)}
\newcommand{\Hdivo}{\bm{H}(\Div,\om)}

\newcommand{\bsh}{\bm{\sigma}_{\calT}}
\newcommand{\btha}{\bm{\tau}_\ver}
\newcommand{\bth}{\bm{\tau}_{\calT}}
\newcommand{\wbth}{\widetilde{\bm{\tau}}_{\calT}}

\newcommand{\gha}{g_{\ver}}
\newcommand{\Pp}{\calP_p}

\newcommand{\Vha}{\bm{V}_{p}(\oma)}

\newcommand{\Pihp}{\Pi_{\calT}^{p}}

\newcommand{\Ihp}{\bm{I}_{\calT}^{p}}
\newcommand{\omf}{\omega_{F}}

\newcommand\eal{{\em et al.}}
\newcommand\ie{i.e.}
\newcommand\cf{cf.}
\newcommand\eg{e.g.}
\newcommand\eq{\coloneqq}
\newcommand\pt{\partial}

\newcommand{\GDa}{\Gamma_{\mathrm D}^\ver}
\newcommand{\LO}{h_\Omega}

\ifIMA\AtBeginDocument{
  \label{CorrectFirstPageLabel}
  
}
\fi

\begin{document}

\ifIMA
\else
\title{Equivalence of local- and global-best approximations,\\ a simple stable local commuting projector,\\ and optimal $hp$ approximation estimates in $\Hdiv$\thanks{This project has received funding from the European Research Council (ERC) under the European Union’s Horizon 2020 research and innovation program (grant agreement No 647134 GATIPOR).}}
\fi

\ifIMA
\else

\author{Alexandre~Ern\footnotemark[2] \footnotemark[3] \and Thirupathi~Gudi\footnotemark[4] \and Iain~Smears\footnotemark[5] \and Martin~Vohral\'ik\footnotemark[3] \footnotemark[2]}

\renewcommand{\thefootnote}{\fnsymbol{footnote}}
\footnotetext[2]{Universit\'e Paris-Est, CERMICS (ENPC), 77455 Marne-la-Vall\'ee, France.}
\footnotetext[3]{Inria, 2 rue Simone Iff, 75589 Paris, France.\\ \indent Emails: \href{mailto:alexandre.ern@enpc.fr}{\texttt{alexandre.ern@enpc.fr}}, \href{mailto:martin.vohralik@inria.fr}{\texttt{martin.vohralik@inria.fr}}.}
\footnotetext[4]{Department of Mathematics, Indian Institute of Science Bangalore, India 560012. Email: \href{mailto:gudi@iisc.ac.in}{\texttt{gudi@iisc.ac.in}}.}
\footnotetext[5]{Department of Mathematics, University College London, London WC1E 6BT, UK. Email: \href{mailto:i.smears@ucl.ac.uk}{\texttt{i.smears@ucl.ac.uk}}.}

\fi


\date{\today}

\maketitle

\begin{abstract}
{Given an arbitrary function in $\Hdiv$, we show that the error attained by the global-best approximation by $\Hdiv$-conforming piecewise polynomial Raviart--Thomas--N\'ed\'elec elements under additional constraints on the divergence and normal flux on the boundary, is, up to a generic constant, equivalent to the sum of independent local-best approximation errors over individual mesh elements, without constraints on the divergence or normal fluxes.
The generic constant only depends on the shape-regularity of the underlying simplicial mesh, the space dimension, and the polynomial degree of the approximations.
The analysis also gives rise to a stable, local, commuting projector in $\Hdiv$, delivering an approximation error that is equivalent to the local-best approximation.
We next present a variant of the equivalence result, where robustness of the constant with respect to the polynomial degree is attained for unbalanced approximations.
These two results together further enable us to derive rates of convergence of global-best approximations that are fully optimal in both the mesh size $h$ and the polynomial degree $p$, for vector fields that only feature elementwise the minimal necessary Sobolev regularity.
We finally show how to apply our findings to derive optimal a priori $hp$-error estimates for mixed and least-squares finite element methods applied to a model diffusion problem.}
\end{abstract}

\ifIMA
\else


\fi


\section{Introduction}\label{sec:Intro}

Interpolation operators that approximate a given function with weak gradient, curl, or divergence by a piecewise polynomial of degree $p$ are fundamental in numerical analysis. Typically, this has to be done over a computational domain $\Omega$ covered by a mesh $\calT$ with characteristic size $h$. Probably the most widespread are the canonical interpolation operators associated with the canonical degrees of freedom of the finite elements from the discrete de Rham sequence, which in particular include the N\'ed\'elec and Raviart--Thomas finite elements.
The advantage of these operators is that they are local (that is, defined independently on each element $K$ of the mesh $\calT$) and that they commute with the appropriate differential operators. They are also projectors, \ie, they leave the interpolated function invariant if it is already a piecewise polynomial, and they lead to optimal approximation error bounds with respect to the mesh size $h$.
However, the canonical interpolation operators have two main deficiencies. Firstly, these operators can act on a given function only if it possesses more regularity beyond the minimal $H^1$, $\Hdiv$, and $\Hcurl$ regularity. Secondly, they are not well-suited to derive approximation error bounds that are quasi-optimal in the polynomial degree $p$.

\subsection{Interpolation operators and $hp$-approximation}

The projection-based interpolation operators, see \ifIMA\else Demkowicz and Buffa~\fi\cite{Demk_Buf_q_opt_proj_int_05}, \ifIMA\else Demkowicz~\fi\cite{Demk_pol_ex_seq_int_06}, and the references therein, lead to {\em optimal approximation properties} in the {\em mesh size $h$} and {\em quasi-optimal approximation properties} in the {\em polynomial degree $p$} (up to logarithmic factors). They were derived under a conjecture of existence of commuting and polynomial-preserving extension operators from the boundary of the given element $K$ to its interior which was later established \ifIMA\else by Demkowicz~\eal\ \fi in~\cite{Demk_Gop_Sch_ext_I_09, Demk_Gop_Sch_ext_II_09, Demk_Gop_Sch_ext_III_12}; the approximation results are summarized in~\cite[Theorem~8.1]{Demk_Gop_Sch_ext_III_12}.
Thus, these operators essentially lift the second drawback of the canonical interpolation operators described above (up to logarithmic factors), while still sharing the same important properties, \ie, they are defined locally, they are projectors, and they commute with the appropriate differential operators. However, these operators again require more regularity beyond the minimal $H^1$, $\Hdiv$, and $\Hcurl$ regularity, so that the first drawback remains.

In the particular case of $\Hdiv$, which constitutes the focus of the present work, the normal component of the interpolate on each mesh face is fully dictated by the normal component of the interpolated function on that face, which requires $\Hsdiv$ regularity with $s > 0$, which is slightly more than $\Hdiv$ regularity.
Some further refinements can be found in \ifIMA\else Bespalov and Heuer~\fi\cite{Besp_Heue_H_div_hp_2D_11} and \ifIMA\else Ern and Guermond~\fi\cite{Ern_Guer_min_reg_nonconf_18}.
Recently, building on~\cite{Demk_Buf_q_opt_proj_int_05, Demk_pol_ex_seq_int_06}, a commuting projector that fully removes the second drawback above in that it has fully {\em optimal $p$-approximation properties} (does not feature the logarithmic factors) has been devised by \ifIMA\else Melenk and Rojik in~\fi\cite{Mel_Roj_com_p_interp_20}. To define the projector, though, higher regularity is needed, with in particular $\Hsdiv$, $s \geq 1$, in the case of interest here.

The issue of constructing (quasi-)interpolation projectors under the minimal regularities $H^1$, $\Hdiv$, and $\Hcurl$ has been addressed before, \cf, \eg, \ifIMA\else Cl\'ement~\fi\cite{Clement_appr_75}, \ifIMA\else Scott and Zhang~\fi\cite{Scott_Zh_int_nonsm_90}, and \ifIMA\else Bernardi and Girault~\fi\cite{Ber_Gir_loc_reg_98} in the $H^1$ case, \ifIMA\else Nochetto and Stamm~\fi\cite{Noch_Stamm_H_dv_19} in the $\Hdiv$ case, and \ifIMA\else Bernardi and Hecht~\fi\cite{Ber_Hecht_appr_Ned_07} in the $\Hcurl$ case; see also the references therein.
Stability and $h$-optimal approximation estimates in any $L^p$-norm, $1 \leq p \leq \infty$,  has recently been achieved by \ifIMA\else Ern and Guermond in~\fi\cite{Ern_Guer_quas_int_best_appr_17} in a unified setting for a wide range of finite elements encompassing the whole discrete de Rham sequence.
The arguments used in~\cite{Ern_Guer_quas_int_best_appr_17} are somewhat different from those in the previous references: a projection onto the fully discontinuous (broken) piecewise polynomial space is applied first, followed by an averaging operator to ensure the appropriate $H^1$, $\Hdiv$, or $\Hcurl$ trace continuity. Unfortunately, all of the quasi-interpolation projectors mentioned in this paragraph do not commute with the appropriate differential operators and, moreover, they are only shown to be optimal in $h$ but not in $p$.

\subsection{Stable local commuting projectors under minimal regularity}

Constructing projectors applicable under the {\em minimal regularities} $H^1$, $\Hdiv$, and $\Hcurl$ that would in addition be {\em commuting}, {\em stable}, and {\em locally defined} represents a long-standing effort. Stability, commutativity, and the projection property were obtained by \ifIMA\else Christiansen and Winther in~\fi\cite{Christ_Wint_sm_proj_08} by composing the canonical interpolation operators with mollification, following some earlier ideas in particular from \ifIMA\else Sch\"oberl~\fi\cite{Schob_com_quas_int_MFE_01, Schob_ML_dec_H_curl_05}, cf. also \ifIMA\else Ern and Guermond~\fi\cite{Ern_Guer_mol_de_Rham_16} for a shrinking technique avoiding the need of extensions outside of the domain and \ifIMA\else Licht~\fi\cite{Licht_sm_proj_19} for essential boundary conditions only prescribed on the part of the boundary of $\Omega$.
These operators are, however, not locally defined. This last remaining issue was finally remedied in~\cite{Falk_Winth_loc_coch_14}, where a patch-based construction resembling that of the Cl\'ement operator~\cite{Clement_appr_75} is introduced. However, no approximation properties are discussed, and stability is achieved only in the graph space of the appropriate differential operator, \eg, $\Hdiv$ but not in $\bL^2$ for the case of interest here.

\subsection{Equivalence of local-best and global-best approximations}

In a seemingly rather unconnected recent result, \ifIMA\else Veeser~\fi\cite{Veeser_approx_grads_16} showed that the error in the {\em best approximation} of a given scalar-valued function in $H^1$ by {\em continuous} piecewise polynomials is {\em equivalent} up to a generic constant to that by {\em discontinuous} piecewise polynomials. This result is termed equivalence of global- and local-best approximations.
A predecessor result in the lowest-order case $p=1$ and up to data oscillation can be easily deduced from \ifIMA\else Carstensen~\eal~\fi\cite[Theorem~2.1 and inequalities~(3.2), (3.5), and (3.6)]{Cars_Pet_Sched_comp_FEs_12}, see also the references therein; equivalences between approximations by different numerical methods are studied in~\cite{Cars_Pet_Sched_comp_FEs_12}. A similar result is also given in \ifIMA\else Aurada~\eal~\fi\cite[Proposition 3.1]{Aur_Fei_Kem_Pag_Praet_loc_glob_13}, and an improvement of the dependence of the equivalence constant on the polynomial degree in two space dimensions is developed in~\cite[Theorem~4]{Can_Noch_Stev_Ver_hp_AFEM_17}.
This equivalence result might be surprising at a first glance, since the local-best error is clearly smaller than the global-best one. The twist comes from the fact that the function to be approximated is continuous in the sense of traces because of its $H^1$-regularity, so one does not gain in approximating it by discontinuous piecewise polynomials.
For finite element discretizations of coercive problems, this result in particular allows one to obtain a priori error estimates without the passage through the Bramble--Hilbert lemma, see \ifIMA\else Gudi~\fi\cite{Gud_a_pr_med_10} or \ifIMA\else Carstensen and Schedensack~\fi\cite{Carst_Sched_medius_elast_14} for important examples of using of a posteriori tools in a priori error analysis.
Another important application is for approximation classes in the theory of a-posteriori-based convergence and optimality~\cite{Veeser_approx_grads_16}.


\subsection{Main results of the manuscript} \label{sec_main_res}

Our main results can be divided into three parts.

\paragraph{1) A simple stable local commuting projector defined under the minimal $\Hdiv$ regularity}

The starting point of our work involves the definition of a projector that maps functions from $\Hdivzero$ (see Section~\ref{sec:Setting} for precise definitions) into the $\Hdivzero$-conforming Raviart--Thomas--N\'ed\'elec space of order $p \geq 0$. This projector enjoys a commuting property with the divergence operator, is locally defined over patches of elements, and is stable in $\bL^2$ up to a $hp$ data oscillation term for the divergence.
Moreover, our projector has a very simple construction, with elementwise local-best approximations combined patch by patch to the final projector via the flux equilibration technique. The essential (no-flux) boundary condition on only a part of the computational domain is here taken into account without any difficulty.
By combining the local-best approximations in a stable manner, the projector achieves, on each element, an error equivalent to local-best errors over a patch of neighbouring elements. All these results are summarized in Definition~\ref{def:proj} and Theorem~\ref{thm:proj} below.

Our main tool for defining the projector is the {\em equilibrated flux reconstruction}. This allows us to transform locally (on patches of elements) a  discontinuous piecewise polynomial with a suitable patchwise divergence constraint into a $\Hdiv$-conforming piecewise polynomial with the expected elementwise divergence constraint.
This has been traditionally used in {\em a posteriori} error analysis of {\em primal} finite element methods derived from $H^1$-formulations, see \ifIMA\else Destuynder and M{\'e}tivet~\fi\cite{Dest_Met_expl_err_CFE_99}, \ifIMA\else Luce and Wohlmuth~\fi\cite{Luce_Wohl_local_a_post_fluxes_04}, \ifIMA\else Braess and Sch{\"o}berl~\fi\cite{Braess_Scho_a_post_edge_08}, \ifIMA\else Ern and Vohral\'ik~\fi\cite{Ern_Voh_adpt_IN_13, Ern_Voh_p_rob_15}, \ifIMA\else Becker~\eal~\fi\cite{Beck_Cap_Luce_flux_rec_16}, and the references therein.
We now employ it here in the context of {\em a priori} error analysis of {\em dual} approximations in $\Hdiv$. Variable polynomial degrees can be taken into account by proceeding as in, \eg, \cite{Dol_Ern_Voh_hp_16}. We avoid it here for the sake of clarity of exposition.

\paragraph{2) Equivalence of local- and global-best approximations in $\Hdiv$ under minimal regularity}

For an arbitrary function in $\Hdivzero$, we consider its global best-approximation error by $\HdivO$-conforming Raviart--Thomas--N\'ed\'elec (RTN) elements of order $p$, defined as the minimum error in a dimensionally consistent weighted $\Hdiv$-norm defined in~\eqref{eq:global_local_equiv_1} below, subject to constraints on the divergence and on the boundary.
In~Theorem~\ref{thm:loc_glob}, we show that the {\em global} best-approximation error is, up to a generic constant, {\em equivalent} to the {\em local}-best approximation errors defined by elementwise minimizations, {\em without any constraint} on the inter-element continuity of the normal trace or on the divergence.
This actually results from the properties of the above projector.
The generic constant entering the equivalence result only depends on the shape-regularity of the simplicial mesh $\calT$, the space dimension $d$, and the polynomial degree $p$.
This extends the results of~\cite{Aur_Fei_Kem_Pag_Praet_loc_glob_13, Can_Noch_Stev_Ver_hp_AFEM_17, Cars_Pet_Sched_comp_FEs_12, Veeser_approx_grads_16} to the $\Hdiv$ case, where we are importantly also able to remove the divergence constraint.

\paragraph{3) Optimal $hp$-approximation estimates in $\Hdiv$}

Our third main result is Theorem~\ref{thm:appr_est} where we derive $hp$-{\em approximation estimates}. These estimates feature the following four properties: i) they request {\em no global regularity} of the approximated function $\bv$ beyond $\Hdivzero$; ii) only the {\em minimal local (elementwise)} $\bH^{s}$-{\em regularity}, $s\geq 0$, is needed; iii) the convergence rates are {\em fully optimal} in both the {\em mesh-size} $h$ and the {\em polynomial degree} $p$, in particular featuring no logarithmic factor of the polynomial degree $p$; iv) {\em no higher-order norms of the divergence} of $\bv$ appear in the bound whenever $s \geq 1$.
This improves on~\cite{Demk_Buf_q_opt_proj_int_05, Demk_pol_ex_seq_int_06} in removing the suboptimality with respect to the polynomial degree, on~\cite{Demk_Buf_q_opt_proj_int_05, Demk_pol_ex_seq_int_06, Mel_Roj_com_p_interp_20} in reducing the regularity requirements, and on approximations using Cl\'ement-type operators in removing the need for regularity assumptions over the (overlapping) elemental patches while reducing it instead to (nonoverlapping) elements.
The proof of these fully optimal $hp$-approximation estimates relies on the elementwise local-best approximation errors of Theorem~\ref{thm:loc_glob} described in point 2) together with its unbalanced but polynomial-degree-robust variant that we develop in Proposition~\ref{prop:loc_glob_p+1}.

\subsection{Applications to mixed finite element and least-squares mixed finite element methods}

The above results can be immediately turned into fully optimal $hp$ {\em a priori error estimates} for two popular classes of numerical methods for second-order elliptic partial differential equations.
In mixed finite element methods, \cf\ the original contributions of \ifIMA\else Raviart and Thomas~\fi\cite{Ra_Tho_MFE_77} and \ifIMA\else N{\'e}d{\'e}lec~\fi\cite{Ned_mix_R_3_80}, or the textbook by \ifIMA\else Boffi~\eal~\fi\cite{Bof_Brez_For_MFEs_13}, the error $\norm{\bsig-\bsig_{\mathrm{M}}}$ between the exact flux $\bsig$ and its mixed approximation $\bsig_{\mathrm{M}}$ immediately takes the form of the $\bL^2$-norm term in the constrained global-best approximation error of Theorems~\ref{thm:loc_glob} and~\ref{thm:appr_est} here (\cf~Lemma~\ref{lem:a_pr_DM}), so the application of our results is immediate. For the family of least-squares mixed finite element methods, see \ifIMA\else Pehlivanov~\eal~\fi\cite{Peh_Car_Laz_LS_MFEs_94},
\ifIMA\else Cai and Ku~\fi\cite{Cai_Ku_LS_MFE_L2_10}, \ifIMA\else Ku~\fi\cite{Ku_LS_MFE_min_reg_13}, and the references therein, the application is a little less immediate, and for completeness we establish it in Lemmas~\ref{lem:a_pr_LSM} and~\ref{lem:a_pr_LSM_div}. These results allow us in particular to circumvent the typical use of interpolation or quasi-interpolation operators to obtain error estimates that hinge upon increased regularity assumptions. Note also that an immediate application of the commuting projector of Definition~\ref{def:proj} in the context of mixed finite elements is the construction of a {\em Fortin operator} under the {\em minimal $\Hdiv$ regularity}.

\subsection{Organization of the manuscript}

The rest of the manuscript is organized as follows. In Section~\ref{sec:Setting}, we introduce the setting and the main notation. In Section~\ref{sec:Meain-results-Statements}, we state our main results, namely Theorem~\ref{thm:proj} about the simple stable local commuting projector, Theorem~\ref{thm:loc_glob} stating the relation between the local- and global-best approximations, and Theorem~\ref{thm:appr_est} stating the optimal $hp$-approximation estimates.
We also show there that Theorem~\ref{thm:loc_glob} follows immediately from Theorem~\ref{thm:proj}.
We then respectively prove Theorems~\ref{thm:proj} and~\ref{thm:appr_est} in Sections~\ref{sec:proof_projector} and~\ref{sec:proof_hp}.
Finally, we present an application of our main results to the a priori error analysis of mixed finite element and least-squares mixed finite methods in Section~\ref{sec:MFEs_a_priori}. A result on polynomial-degree-robust equivalence between constrained and unconstrained best approximations on a simplex is presented in Appendix~\ref{app:constr}; it is of independent interest.

\section{Setting and notation}\label{sec:Setting}

\subsection{Domain $\Omega$, space $\Hdivzero$, and simplicial mesh $\calT$}

Let $\Omega\subset \R^d$ for $d\in \{2,3\}$ be an open, bounded, connected polygon or polyhedron with Lipschitz boundary~$\Gamma$. Let $\calT$ be a given conforming, simplicial, possibly locally refined mesh of $\Omega$, i.e. $\overline{\Omega}=\cup_{K \in \calT} K$, where any $K$ is a closed simplex and the intersection of two different simplices is either an empty set or their common vertex, edge, or face.
Let $\GD$ be a (possibly empty) closed subset of $\Gamma$, and let $\GN\coloneqq\Gamma\setminus\GD$ be its (relatively open) complement in $\Gamma$, with the assumption that $\calT$ matches $\GD$ and $\GN$ in the sense that every boundary face of the mesh $\calT$ is fully contained either in $\GD$ or in $\overline{\GN}$.
Let $\bLO\coloneqq L^2(\Omega;\R^\dim)$, and $\HdivO \coloneqq \{\bv \in \bLO, \, \nabla{\cdot}\bv \in L^2(\Omega)\}$.
Furthermore, we define the space $\Hdivzero \coloneqq \{\bv\in \HdivO,\, \bv{\cdot}\bn=0\text{ on } \Gamma_{\mathrm{N}}\}$, where $\bv{\cdot}\bn=0$ on $\Gamma_{\mathrm{N}}$ means that $\pair{\bv{\cdot}\bn}{\varphi}_{\Gamma}=0$ for all functions $\varphi\in H^1(\Omega)$ that have vanishing trace on $\GD$; here $\pair{\bv{\cdot}\bn}{\varphi}_{\Gamma}\coloneqq \int_{\Omega}\left[ \bv{\cdot} \nabla \varphi+(\nabla{\cdot}\bv) \varphi\right]$.
For an open subset $\om \subset \Omega$, let $\bLo \eq L^2(\om;\R^\dim)$ and $\Hdivo \eq \{\bv \in \bLo, \, \nabla{\cdot}\bv \in L^2(\om)\}$.
We also denote by $({\cdot},{\cdot})_{\om}$ and $\norm{\cdot}_{\om}$ the $L^2$-inner product and norm for scalar- or vector-valued functions on $\om$.
In the special case where $\om = \Omega$, we drop the subscript, i.e. $({\cdot},{\cdot}) \eq ({\cdot},{\cdot})_{\Omega}$ and $\norm{{\cdot}}\coloneqq\norm{{\cdot}}_{\Omega}$.
The diameter of $\om$ is denoted by $h_\om$, and its outward unit normal as $\bn_\om$.

\subsection{Elements, vertices, faces, and patches of elements} \label{sec:mesh_not}

For any mesh element $K \in \calT$, its diameter is denoted by $h_K$, and we set $h \coloneqq \max_{K\in\calT}h_K$.
Let $\calVint$ denote the set of interior vertices of $\calT$, i.e.\ the vertices contained in $\Omega$. Let $\calVext$ denote the set of vertices of $\calT$ on the boundary $\Gamma$, and set $\calV\coloneqq\calVint\cup\calVext$.
We divide $\calVext$ into two disjoint sets $\calVextDb$ and $\calVextNb$, where $\calVextDb$ contains all vertices in $\GD$ (recalling that $\GD$ is assumed to be closed) and $\calVextNb$ consists of all vertices in $\GN$.
For each vertex $\ver\in \calV$, define the patch $\Ta\coloneqq\{K\in\calT,\, \ver \text{ is a vertex of } K\}$ and the corresponding open subdomain $\oma \eq \{\cup_{K\in\Ta} K\}^\circ$.
The piecewise affine Lagrange finite element basis function associated with a vertex $\ver\in\calV$ is denoted by $\psia$.
Let $\calF$ denote the set of all $(d-1)$-dimensional faces of $\calT$. By convention, we consider faces to be closed sets.
For an element $K\in\calT$, we denote the set of all faces of $K$ by $\calFK$, and the set of all vertices of $K$ by $\calV_K$.
For each interior vertex $\ver \in \calVint$, we let $\calFaint$ denote the set of all faces that contain the vertex $\ver$ (and thus do not lie on the boundary of $\oma$). For boundary vertices $\ver \in \calVext$, let $\calFaint$ collect the faces that contain the vertex $\ver$ but do not lie on the Dirichlet boundary $\Gamma_D$. The mesh shape-regularity parameter is defined as $\kappa_\calT \eq \max_{K \in \calT} h_K / \varrho_K$, where $\varrho_K$ is the diameter of the largest ball inscribed in $K$.

\subsection{Piecewise polynomial and Raviart--Thomas--N\'ed\'elec spaces}

Let $p\geq 0$ be a nonnegative integer.
For $S\in \{K,F\}$, where $K\in\calT$ is an element and $F\in \calF$ is a face, we define $\Pp(S)$ as  the space of all polynomials of total degree at most $p$ on $S$.
If $\widetilde{\calT}$ denotes a subset of elements of $\calT$, $ \Pp(\widetilde{\calT}) \coloneqq\{ v_h\in L^2(\Omega),\, v_h|_K\in \Pp(K)\;\,\forall K\in\widetilde{\calT}\}$ is
the space of piecewise polynomials of degree at most $p$ over $\widetilde{\calT}$. Typically, $\widetilde{\calT}$ will be either the whole mesh $\calT$ or the patch $\Ta$ as defined above. We define the piecewise Raviart--Thomas--N\'ed\'elec space $\RTN \coloneqq \{ \bvh \in \bLO,\;\bvh|_K \in \RTNK\}$, where $\RTNK \coloneqq \Pp(K;\R^\dim) + \bx \Pp(K)$ and $\calP_p(K;\R^\dim)$ denotes the space of $\R^\dim$-valued functions defined on $K$ with each component being a polynomial of degree at most $p$ in $\Pp(K)$.
Note that with this choice of notation, functions in the space $\RTN$ do not necessarily belong to $\HdivO$; thus, $\RTN \cap \HdivO$ is a proper subspace of $\RTN$ which is classically characterized as those functions in $\RTN$ having a continuous normal component across interior mesh faces.
Moreover, $\RTN$ is a subspace of $\bC^1(\calT) \coloneqq\{\bv\in \bLO, \; \bv|_K \in \bC^1(K) \text{ for all } K\in \calT\}$, the space of piecewise (broken) first-order component-wise differentiable vector-valued fields over $\calT$.
To avoid confusion between piecewise smooth and globally smooth functions, we denote the elementwise gradient and the elementwise divergence by $\nabla_{\calT}$ and by $\nabla_{\calT}{\cdot}$, respectively.

\subsection{$L^2$-orthogonal projection and elementwise canonical interpolant}

For each polynomial degree $p\geq 0$, let $\Pihp \colon L^2(\Omega)\tends \Pp(\calT)$ denote the $L^2$-orthogonal projection of order $p$.
Similarly, let $\Pi_F^p$ denote the $L^2$-orthogonal projection of order $p$ on a face $F\in \calF$, which maps $L^2(F)$ to $\Pp(F)$.
Let $\Ihp \colon \bC^1(\calT) \tends \RTN$ be the {\em elementwise} canonical (Raviart--Thomas--N\'{e}d\'{e}lec) interpolant. The domain of $\Ihp$ can be taken (much) larger than $\bC^1(\calT)$, but not as large as piecewise $\Hdiv$ fields; the present choice is sufficient for our purposes. For any $\bv\in \bC^1(\calT)$, the interpolant $\Ihp \bv$ is defined separately on each element $K\in \calT$ by the conditions
\begin{equation}\label{eq:Ihp_def}
\begin{aligned}
( (\Ihp \bv)|_K {\cdot} \bn_{K} , q_K)_F &= (\bv|_{K} {\cdot} \bn_{K}, q_K)_F &&\forall q_K\in \Pp(F),\;\forall F\in \calF_K, \\
( \Ihp \bv, \br_K)_K & = (\bv, \br_K)_K &&\forall \br_K \in \calP_{p-1}(K;\R^\dim),
\end{aligned}
\end{equation}
where $\bv|_{K} {\cdot} \bn_{K}$ denotes the normal trace of $\bv|_K$, the restriction of $\bv$ to $K$. Note that~\eqref{eq:Ihp_def} implies that $((\Ihp \bv)|_K{\cdot} \bn_{K})|_F = \Pi_F^p ((\bv|_K{\cdot} \bn_{K})|_F)$ for all faces $F\subset \calFK$.
A useful property of the operator~$\Ihp$ is the commuting identity:
\begin{equation}\label{eq:Ihp_identity}
\begin{aligned}
\nabla_{\calT}{\cdot} (\Ihp\bv) = \Pihp(\nabla_{\calT}{\cdot}\bv) &&&\forall\bv\in\bC^1(\calT).
\end{aligned}
\end{equation}

\subsection{Spaces for patchwise equilibration}

In the spirit of \ifIMA\else Braess~\eal~\fi\cite{Brae_Pill_Sch_p_rob_09} and~\cite{Ern_Voh_adpt_IN_13, Ern_Voh_p_rob_15, Ern_Smears_Voh_H-1_lift_17}, we finally define the local mixed finite element spaces $\Vha$ by
\begin{equation}\label{eq:Vha_def}
\begin{aligned}
\Vha & \coloneqq
\begin{cases}
   \left\{\bva \in \RTNa\cap \Hdiva ,\, \bva{\cdot} \bn_{\oma} =0\text{ on }\p\oma \right\} & \quad \text{if }\ver\in\calVint\cup\calVextNb,\\
    \left\{\bva \in\RTNa\cap  \Hdiva ,\, \bva{\cdot} \bn_{\oma} =0\text{ on }\p\oma\setminus\GDa  \right\}& \quad \text{if }\ver\in\calVextDb,
\end{cases}
\end{aligned}
\end{equation}
where $\GDa$ contains those boundary faces from $\GD$ that share the vertex $\ver$.
In particular, we observe that when $\partial\oma \cap\GN\neq \emptyset$, then $\bva{\cdot}\bn =0$ on $\GN$ for any $\bva\in\Vha$.
As a result of the above definitions, it follows that the zero extension to all of $\Omega$ of any  $\bva\in \Vha$ belongs to $\RTN\cap\Hdivzero$.

\section{Main results}\label{sec:Meain-results-Statements}

This section collects our main results.

\subsection{A simple stable local commuting projector in $\Hdivzero$}\label{sec_proj}

Our first main result is a construction of a simple, locally defined, and stable commuting projector defined over the entire $\Hdivzero$ that leads to an approximation error that is equivalent to the local-best approximation error.

Recall the definition of the broken Raviart--Thomas--N\'{e}d\'{e}lec interpolant $\Ihp$ from~\eqref{eq:Ihp_def} and that of the piecewise polynomial patchwise $\Hdiva$-conforming spaces $\Vha$ from~\eqref{eq:Vha_def}. Recall also that zero extensions of elements of $\Vha$ belong to $\RTN\cap \Hdivzero$, and that $\psia$ is the piecewise affine Lagrange finite element basis function associated with the vertex $\ver$.

\begin{definition}[A simple locally-defined mapping from $\Hdivzero$ to $\RTN$ $\cap \Hdivzero$] \label{def:proj}
Let $\bv \in \Hdivzero$ be arbitrary. Let $\bth \in \RTN$ be defined elementwise by
\begin{equation}\label{eq_L2_proj_constraint}
    \bth|_K \eq \argmin_{\substack{\bv_K \in \RTNK \\ \nabla{\cdot}\bv_K = \Pihp(\nabla{\cdot}\bv)}} \norm{\bv-\bv_K}_K \qquad \forall K \in \calT.
\end{equation}
For each mesh vertex $\ver\in \calV$, let $\saa \in \Vha$ be defined by
\begin{equation}\label{eq:sha_def}
\saa \coloneqq \argmin_{\substack{\bva \in \Vha \\ \nabla{\cdot}\bva= \Pihp (\psia \nabla{\cdot} \bv +\nabla \psia {\cdot} \bth)}}
\norm{\bva-\Ihp(\psia\bth)}_{\oma}.
\end{equation}
Extending the functions $\saa$ from the patch domains $\oma$ to the rest of $\Omega$ by zero, we define $P^p_{\calT} (\bv) \in \RTN \cap \Hdivzero$ by
\begin{equation}\label{eq:bsh_def}
P^p_{\calT} (\bv) \eq \bsh \eq \sum_{\ver\in\calV}\saa.
\end{equation}
\end{definition}

The justification that the construction of $P^p_{\calT}(\bv)$ is well-defined is given in Section~\ref{sec:exist_Pp} below.
The first step~\eqref{eq_L2_proj_constraint} in Definition~\ref{def:proj} considers the elementwise $\bL^2$-norm local-best approximation that defines the {\em discontinuous piecewise RTN polynomial} $\bth$ closest to $\bv$ under the divergence constraint.
The second step in~\eqref{eq:sha_def} can be seen as {\em smoothing} $\bth$ over the patch subdomains $\oma$ to obtain an $\Hdiv$-conforming approximation $\saa$ over each vertex patch with a suitably prescribed divergence. These approximations $\saa$ are then summed into $P^p_{\calT}(\bv)$.
The overall procedure is motivated by {\em equilibrated flux reconstructions} coming from a posteriori error estimation~\cite{Dest_Met_expl_err_CFE_99, Braess_Scho_a_post_edge_08, Ern_Voh_adpt_IN_13}. Here we adapt those techniques to the purpose of a priori error analysis.

Our first main result, whose proof is postponed to Section~\ref{sec:proof_projector}, is the following.

\begin{theorem}[Commutativity, projection, approximation, and stability of $P^p_{\calT}$]\label{thm:proj} Let a mesh $\calT$ of $\Omega$ and a polynomial degree $p \geq 0$ be fixed. Then, the operator~$P^p_{\calT}$ from Definition~\ref{def:proj} maps $\Hdivzero$ to $\RTN\cap \Hdivzero$ and
\begin{align}
\nabla{\cdot}P^p_{\calT} (\bv)&=  \Pihp (\nabla{\cdot} \bv)  && \forall \bv \in \Hdivzero,\label{eq:commuting}\\
P^p_{\calT} (\bv) &= \bv  && \forall \bv \in \RTN \cap \Hdivzero. \label{eq:proj}
\end{align}
Thus $P^p_{\calT}$ is a projection from $\Hdivzero$ onto $\RTN\cap \Hdivzero$ that commutes with the divergence.
Furthermore, for any $\bv \in \Hdivzero$ and any $K\in\calT$, we have the approximation and stability bounds
\begin{align}
\norm{\bv-P^p_{\calT} (\bv)}^2_K + & \Big[\frac{h_K}{p+1}\norm{\nabla{\cdot}(\bv-P^p_{\calT} (\bv))}_K\Big]^2  \label{eq:proj_approx} \\
\leq {} & C \sum_{K^{'\!\!}\in\calT_K} \Bigg\{\min_{\bv_{K^{'\!\!}}\in \RTNKp}\norm{\bv - \bv_{K^{'\!\!}}}_{K^{'\!\!}}^2 + \bigg[\frac{h_{K^{'\!\!}}}{p+1}\norm{\nabla{\cdot}\bv - \Pihp(\nabla{\cdot}\bv)}_{K^{'\!\!}}\bigg]^2\Bigg\},\nonumber\\
\norm{P^p_{\calT} (\bv)}_K^2 \leq {} & C \sum_{K^{'\!\!} \in \calT_K} \Bigg\{\norm{\bv}_{K^{'\!\!}}^2 + \Big[\frac{h_{K^{'\!\!}}}{p+1} \norm{\nabla{\cdot} \bv - \Pihp (\nabla{\cdot} \bv)}_{K^{'\!\!}}\Big]^2\Bigg\}  ,\label{eq:stab} \\
\norm{P^p_{\calT} (\bv)}_K^2+h_\Omega^2\norm{\nabla{\cdot} P^p_{\calT} (\bv)}_K^2 \leq {} & C \sum_{K^{'\!\!} \in \calT_K}\Big\{\norm{\bv}_{K^{'\!\!}}^2+h_\Omega^2\norm{\nabla{\cdot} \bv}_{K^{'\!\!}}^2\Big\} , \label{eq:stab_div}
\end{align}
where $\calT_K\coloneqq \cup_{\ver\in \calV_K}\Ta$ are the neighboring elements of $K$, and recalling that $h_\Omega$ denotes the diameter of $\Omega$. The constant $C$ above only depends on the space dimension $\dim$, the shape-regularity parameter $\kappa_\calT$ of~$\calT$, and the polynomial degree $p$.
\end{theorem}

Property~\eqref{eq:stab} readily implies that $P^p_{\calT}$ is globally $L^2$-stable up to $hp$ data oscillation of the divergence, since summing over the mesh elements leads to
\begin{align}\label{eq_stab}
\norm{P^p_{\calT} (\bv)}^2 \leq C \Bigg\{\norm{\bv}^2 +\sum_{K \in \calT} \Big[\frac{h_K}{p+1} \norm{\nabla{\cdot} \bv - \Pihp (\nabla{\cdot} \bv)}_K\Big]^2\Bigg\} \qquad\forall \bv \in \Hdivzero.
\end{align}
Similarly, from~\eqref{eq:stab_div}, we infer that $P^p_{\calT}$ is  $\Hdiv$-stable, since
\begin{align*}
\norm{P^p_{\calT} (\bv)}^2+h_\Omega^2\norm{\nabla{\cdot} P^p_{\calT} (\bv)}^2\leq C \big[\norm{\bv}^2+h_\Omega^2\norm{\nabla{\cdot} \bv}^2\big] \qquad\forall \bv \in \Hdivzero.
\end{align*}
The projector $P^p_{\calT}$ in Definition~\ref{def:proj} and Theorem~\ref{thm:proj} improves on~\cite{Christ_Wint_sm_proj_08} in that the construction is {\em local}, and on~\cite{Falk_Winth_loc_coch_14} in that it is {\em stable in $\bL^2$}, up to data oscillation, see~\eqref{eq_stab}, rather than only in $\Hdiv$. We note that, for the divergence term, \eqref{eq:stab} improves the bound~(5.2) of~\cite[Theorem~5.2]{Falk_Winth_loc_coch_14} since, in particular, we have $\norm{\nabla{\cdot} \bv - \Pihp (\nabla{\cdot} \bv)}_K$ in place of $\norm{\nabla{\cdot} \bv}_K$, whereas~\eqref{eq:stab_div} is similar to the combination of the bounds~(5.2) and~(5.3) of~\cite[Theorem~5.2]{Falk_Winth_loc_coch_14}. The projection operator $P^p_{\calT}$ defined here also satisfies the {\em commuting property} with the divergence operator~\eqref{eq:commuting}, in contrast to~\cite{Ern_Guer_quas_int_best_appr_17}.

\subsection{Equivalence of local- and global-best approximations in $\Hdivzero$} \label{sec:equiv}

For any function $\bv\in \Hdivzero$, we consider the {\em global-best approximation error} $\Eglob(\bm{v})$ defined as the best approximation, in a weighted norm, from $\RTN$ $\cap \Hdivzero$, subject to a constraint on the divergence:
\begin{equation}\label{eq:global_local_equiv_1}
[\Eglob(\bm{v})]^2 \coloneqq \min_{\substack{\bvh\in \RTN\cap \Hdivzero\\ \nabla{\cdot}\bvh=\Pihp(\nabla{\cdot}\bv)}} \norm{\bv-\bvh}_{\Omega}^2 +
\sum_{K\in\calT}\bigg[\frac{h_K}{p+1}\norm{\nabla{\cdot}\bv - \Pihp(\nabla{\cdot}\bv)}_K\bigg]^2.
\end{equation}
We further consider the {\em local-best approximation errors} defined on each element $K\in \calT$ by
\begin{equation}\label{eq:local_minimizers}
[\Eloc(\bm{v})]^2 \coloneqq  \min_{\bv_K\in \RTNK}\norm{\bv - \bv_K}_K^2 + \bigg[\frac{h_K}{p+1}\norm{\nabla{\cdot}\bv - \Pihp(\nabla{\cdot}\bv)}_K\bigg]^2.
\end{equation}
Note that the minimization in~\eqref{eq:local_minimizers} does {\em not involve} a constraint on the divergence nor on the normal component on $\GN$ (whenever relevant). Furthermore, since $\Pihp$ is the $L^2$-orthogonal projection onto the broken polynomial space $\Pp(\calT)$, we have $\norm{\nabla{\cdot}\bv - \Pihp(\nabla{\cdot}\bv)}_K = \min_{q \in \calP_p(K)}\norm{\nabla{\cdot}\bv - q}_K$.
Thus the local approximation errors $\Eloc(\bv)$ involve the local-best approximation errors in $\bL^2$ plus a weighted $L^2$ best approximation error of the divergence.

In a direct consequence of Theorem~\ref{thm:proj}, we now show that the global-best error $\Eglob(\bv)$ is in fact {\em equivalent} to the root-mean square sum of the local-best errors $\Eloc(\bv)$ over all elements of the mesh.

\begin{theorem}[Equivalence of local- and global-best approximations]\label{thm:loc_glob}
There exists a  constant~$C$ depending only on the space dimension $\dim$, the shape-regularity parameter $\kappa_\calT$ of $\calT$, and the polynomial degree $p\ge0$, such that, for any $\bv\in \Hdivzero$,
\begin{equation}\label{eq:global_local_equiv_2}
  \left[\Eglob(\bv)\right]^2 \leq  C \sum_{K\in\calT}\left[\Eloc(\bv)\right]^2 \leq C \left[\Eglob(\bv)\right]^2.
\end{equation}
\end{theorem}

\begin{proof}
Consider an arbitrary function $\bv\in \Hdivzero$; then Theorem~\ref{thm:proj} shows that the projection $P^p_{\calT}(\bv)\in \RTN\cap \Hdivzero$ satisfies the constraints of the global minimization set in~\eqref{eq:global_local_equiv_1} due to its commutativing property \eqref{eq:commuting}. Therefore, the first inequality in~\eqref{eq:global_local_equiv_2} follows by picking the function $P^p_{\calT}(\bv)$ from the minimization set, summing the bound in the  local approximation property~\eqref{eq:proj_approx} over all mesh elements, and invoking the shape-regularity of the mesh which implies that the number of neighbors a mesh cell can have is uniformly bounded from above. Meanwhile, the second inequality in~\eqref{eq:global_local_equiv_2} follows straightforwardly from the definitions in~\eqref{eq:global_local_equiv_1} and~\eqref{eq:local_minimizers}.
\end{proof}

\begin{remark}[Necessity of the divergence error terms]
Although the scaled divergence terms $\frac{h_K}{p+1}\norm{\nabla{\cdot}\bv - \Pihp(\nabla{\cdot}\bv)}_K$ take an identical form in both $\Eglob(\bv)$ and $\Eloc(\bv)$, they cannot be removed from the local contributions $\Eloc(\bv)$. Otherwise, it would be possible to choose a sequence of functions $\bv$ in $\Hdivzero$ approaching a function $\bth \in \RTN$ but $\bth \notin \Hdivzero$ such that the middle term in~\eqref{eq:global_local_equiv_2} would tend to zero but $\Eglob(\bv)$ would remain uniformly bounded away from zero.
\end{remark}

\begin{remark}[Equivalence with constraint on the right-hand side]\label{rem:constr_rhs}
Theorem~\ref{thm:loc_glob} also straightforwardly implies that
\begin{align*}
\left[\Eglob(\bv)\right]^2 & \leq C \sum_{K\in \calT}\left\{ \min_{\substack{\bv_K\in \RTNK\\\nabla{\cdot} \bv_K=\Pihp(\nabla{\cdot}\bv)|_K}}
\norm{\bv-\bv_K}_K^2+\bigg[\frac{h_K}{p+1}\norm{\nabla{\cdot}\bv - \Pihp(\nabla{\cdot}\bv)}_K\bigg]^2\right\} \\
& \leq C \left[\Eglob(\bv)\right]^2
\end{align*}
with the same constant $C$, where the minimization problems in the middle term include a constraint on the divergence to mirror the divergence constraint in $\Eglob(\bv)$.
\end{remark}

\subsection{Optimal-order $hp$-approximation estimates in $\Hdivzero$}

We finally focus on functions with some additional elementwise regularity.
For any $s\geq 0$ and any mesh element $K\in\calT$, let $\bH^s(K)$ denote the space of vector fields in $\bL^2(K)$ with each component in $H^s(K)$.
Recall the definition~\eqref{eq:global_local_equiv_1} of $\Eglob(\bm{v})$.
Our third and last main result, whose proof is postponed to Section~\ref{sec:proof_hp}, delivers $hp$-{\em optimal} convergence rates for vector fields in $\Hdivzero$ with the {\em minimally} necessary additional elementwise {\em regularity}.

\begin{theorem}[$hp$-optimal approximation estimates under minimal regularity]\label{thm:appr_est}
Let $s\geq 0$ and let $\bv \in \Hdivzero$ be such that
 \[
\bv|_K \in \bH^{s}(K) \quad\forall K\in\calT.
 \]
Let the polynomial degree $p\geq 0$. Then there exists a constant $C$, depending only on the regularity exponent $s$, the space dimension $\dim$, and the shape-regularity parameter $\kappa_\calT$ of $\calT$, such that
\begin{equation}\label{eq:est_p+1}
[\Eglob(\bv)]^2 \leq C  \bigg\{
\sum_{K\in\calT}\Big[\frac{h_K^{\min(s,p+1)}}{(p+1)^{s}}\norm{\bv}_{\bH^{s}(K)}\Big]^2 + \delta_{s<1}\Big[\frac{h_K}{p+1} \norm{\nabla{\cdot}\bv}_K\Big]^2 \bigg\},
\end{equation}
where $\delta_{s<1}\coloneqq1$ if $s<1$ and $\delta_{s<1}\coloneqq0$ if $s\ge1$.
\end{theorem}

\section{Proof of Theorem~\ref{thm:proj} (commutativity, projection, approximation, and stability of $P^p_{\calT}$)}\label{sec:proof_projector}

The proof of Theorem~\ref{thm:proj} is split into several parts.
First, in Section~\ref{sec:exist_Pp}, we analyse essential properties of the construction of the mapping $P^p_{\calT}$ from Definition~\ref{def:proj}. We next establish the statement~\eqref{eq:commuting} from Theorem~\ref{thm:proj} in Section~\ref{sec:commut}, showing that the operator $P^p_{\calT}$ commutes with the divergence.
Then, in Section~\ref{sec:proof_approx_property}, we prove the statement~\eqref{eq:proj_approx}  from Theorem~\ref{thm:proj} on the approximation properties of $P^p_{\calT}$. This is the most technical part of the proof.
Finally, in Section~\ref{sec:proof_proj_concl}, we conclude by proving the remaining three statements~\eqref{eq:proj}, \eqref{eq:stab}, and \eqref{eq:stab_div} (the projection property, $\bL^2$ stability, and $\Hdiv$ stability).

\subsection{Justification of the construction of $P^p_{\calT}$}\label{sec:exist_Pp}

We start by showing that the operator $P^p_{\calT}$ of Definition~\ref{def:proj} is well-defined on $\Hdivzero$. Recall the notation from Section~\ref{sec:mesh_not}.

\begin{lemma}[Discrete weak divergence of $\bm{L}^2$-projection]\label{lem:rtn_orthogonality}
For any function $\bv\in\Hdivzero$, let $\bth$ be defined elementwise in~\eqref{eq_L2_proj_constraint}. Then
 \begin{equation}\label{eq:rtn_orthogonality}
  \begin{aligned}
  (\nabla{\cdot}\bv, \psia)_{\oma} + (\bth,\nabla \psia)_{\oma} = 0 &&&\forall\ver\in\calVint\cup\calVextNb.
  \end{aligned}
  \end{equation}
\end{lemma}
\begin{proof}
First, observe that for any vertex $\ver\in\calVint\cup\calVextNb$, the hat function $\psia$ belongs to $H^1_{\GD}(\Omega)$ owing to the conformity of $\calT$ with respect to the Dirichlet and Neumann boundary sets. Therefore, $(\nabla{\cdot}\bv, \psia)_{\oma} + (\bv,\nabla \psia)_{\oma}=0$, where we use the fact that $\oma$ is the support of $\psia$. Since $\nabla \psia$ is a constant vector on each element $K$, the Euler--Lagrange equations for~\eqref{eq_L2_proj_constraint} imply that
\begin{equation}\label{eq:orth_K}
  (\bth,\nabla\psia)_K=(\bv,\nabla\psia)_K \qquad \forall K \in \Ta.
\end{equation}
Consequently, $(\bth,\nabla\psia)_{\oma}=(\bv,\nabla\psia)_{\oma}$, and~\eqref{eq:rtn_orthogonality} follows.
\end{proof}

We now show that the local minimization problems~\eqref{eq:sha_def} give well-defined local contributions $\saa$.

\begin{lemma}[Existence and uniqueness of local problems]\label{lem:ex_un}
For each vertex $\ver\in \calV$, there exists a unique $\saa\in \Vha$ satisfying~\eqref{eq:sha_def}.
\end{lemma}
\begin{proof}
The minimization problem~\eqref{eq:sha_def} is equivalent to a mixed finite element problem in the patch subdomain $\oma$.
For Dirichlet boundary vertices $\ver \in \calVextDb$, this problem is well-posed with a unique minimizer since the space $\Vha$ of~\eqref{eq:Vha_def} does not impose the normal constraint everywhere on $\partial \oma$. For interior and Neumann vertices $\ver\in\calVint\cup \calVextNb$, the source term in the divergence constraint satisfies the compatibility condition
\[
    (\Pihp (\psia \nabla{\cdot} \bv + \nabla \psia {\cdot} \bth), 1)_\oma = (\nabla{\cdot}\bv, \psia)_{\oma} + (\bth,\nabla \psia)_{\oma} = 0,
\]
where the second equality follows from~Lemma~\ref{lem:rtn_orthogonality}. Therefore, $\saa$ is also well-defined for interior and Neumann vertices $\ver\in\calVint\cup \calVextNb$.
\end{proof}

It follows from Lemma~\ref{lem:ex_un} that $P^p_{\calT} (\bv) \in \RTN\cap\Hdivzero$ is well-defined for every $\bv\in \Hdivzero$.

\subsection{Proof of the commuting property~\eqref{eq:commuting}} \label{sec:commut}

We are now ready to establish:

\begin{lemma}[Commuting property~\eqref{eq:commuting} from Theorem~\ref{thm:proj}]\label{lem:commuting}
 $P^p_{\calT}$ satisfies~\eqref{eq:commuting}.
\end{lemma}

\begin{proof}
Since the functions $\{\psia\}_{\ver\in\calV}$ form a partition of unity over $\Omega$, \ie, $\sum_{\ver\in\calV}\psia =1$, and consequently $\sum_{\ver\in\calV}\nabla\psia =\bm{0}$, we find that
\begin{equation}\label{eq:equilibrated_proof}
\nabla{\cdot}P^p_{\calT} (\bv) = \sum_{\ver\in\calV}\nabla{\cdot}\saa = \sum_{\ver\in\calV}\big\{\Pihp (\psia \nabla{\cdot} \bv +\nabla \psia {\cdot} \bth)\big\} =\Pihp(\nabla{\cdot}\bv).\qedhere
\end{equation}
\end{proof}

\subsection{Proof of the approximation property~\eqref{eq:proj_approx}}\label{sec:proof_approx_property}

Let us start with two useful technical results.
For a given vertex $\ver \in \calV$, let the space~$H^1_*(\oma)$ be defined
by
\begin{equation} \label{eq:H1s}
H^1_*(\oma) \coloneqq \begin{cases}
 \{\varphi\in H^1(\oma),\quad (\varphi,1)_{\oma} = 0\} &\text{if }\ver \in \calVint \cup \calVextNb , \\
 \{\varphi\in H^1(\oma),\quad \varphi|_{\p \oma\cap \GDa} =0\} &\text{if }\ver\in\calVextDb,
\end{cases}
\end{equation}
where we recall that $\GDa$ contains those boundary faces from $\GD$ that share the vertex $\ver$.
Recall also the discrete spaces $\Vha$ defined in~\eqref{eq:Vha_def}. The following result has been shown in \ifIMA\else Braess~\eal~\fi\cite[Theorem~7]{Brae_Pill_Sch_p_rob_09} in two space dimensions and~\cite[Corollaries~3.3, 3.6, and~3.8]{Ern_Voh_p_rob_3D_20} in three space dimensions.

\begin{lemma}[Stability of patchwise flux equilibration] \label{lem:main_stability_bound}
Let a vertex $\ver\in\calV$ be fixed, and let $\gha\in \Pp(\Ta)$ and $\btha \in \RTNa$ be given discontinuous piecewise polynomials with the condition $(\gha,1)_{\oma} =0$ if $\ver\in\calVint\cup \calVextNb$.
Then, there exists a constant $C$, depending only on the space dimension $\dim$ and the mesh shape-regularity parameter $\kappa_\calT$, such that
\[
\min_{\substack{\bv_{\ver} \in \Vha\\ \nabla{\cdot}\bv_{\ver} = \gha}} \norm{\bv_{\ver}-\btha}_{\oma}
\leq C \sup_{\substack{\varphi\in H^1_*(\oma)\\ \norm{\nabla \varphi}_{\oma}=1}}
\left\{(\gha,\varphi)_{\oma} + (\btha,\nabla \varphi)_{\oma} \right\}.
\]
\end{lemma}

We shall also use the following auxiliary bound for face terms based on the bubble function technique of~Verf{\"u}rth, \cf~\cite{Verf_13}, from a posteriori error analysis.

\begin{lemma}[Bound on face terms]\label{lem:faces}
Let a mesh face $F\in \calF$ be fixed, and let $\calT_F$ be the set of one or two mesh elements $K\in\calT$ to which $F$ belongs, with $\omf$ the corresponding open subdomain. Let $h_F$ denote the diameter of $F$.
Then, there exists a constant $C$, depending on the space dimension $\dim$, the mesh shape-regularity parameter $\kappa_\calT$, and the polynomial degree $p$, such that
\[
 h_F^{1/2}  \norm{q_h}_{F} \leq C \sup_{\substack{\varphi\in H^1(\omf)\\ \varphi=0 \text{ on } \pt \omf \setminus F\\ \norm{\nabla \varphi}_{\omf}=1}} (q_h,\varphi)_F \qquad \forall q_h \in \Pp(F).
\]
\end{lemma}


We are now ready to prove the statement~\eqref{eq:proj_approx} from Theorem~\ref{thm:proj}, where we now employ the short-hand notation $\Eloc(\bm{v})$ from~\eqref{eq:local_minimizers}. Let $\bv \in \Hdivzero$ be arbitrary. Since it follows from $\nabla{\cdot}P^p_{\calT} (\bv) = \Pihp(\nabla{\cdot}\bv)$ that
\[
\frac{h_K}{p+1}\norm{\nabla{\cdot}\bv-\nabla{\cdot}P^p_{\calT} (\bv)}_K \leq \Eloc(\bv),
\]
it only remains to prove that
\begin{equation}\label{eq:proof_main_bound}
\norm{\bv-P^p_{\calT} (\bv)}_{K} \leq C \left\{ \sum_{K^{'\!\!}\in\calT_K} \Eloct(\bv)^2 \right\}^{\frac{1}{2}} \qquad \forall K \in \calT.
\end{equation}
We proceed for this purpose in two steps.

\paragraph{Step~1. Bound on $\saa$.} Recall that $\saa$ is defined in~\eqref{eq:sha_def} with $\bth$ defined elementwise in~\eqref{eq_L2_proj_constraint}.

\begin{lemma}[Bound on $\saa$]\label{lem:bound_sha}
There exists a constant $C$, depending only on $\dim$, $\kappa_\calT$, and $p$, such that
\begin{equation}\label{eq:sha_bound}
\norm{\saa - \Ihp(\psia\bth)}_{\oma} \leq C \left\{ \sum_{K\in\Ta} [\Eloc(\bv)]^2 \right\}^{\frac{1}{2}} \qquad \forall \ver \in \calV.
\end{equation}
\end{lemma}
\begin{proof}
First, since $\Ihp (\psia\bth) \in \RTNa$, we can apply Lemma~\ref{lem:main_stability_bound} to $\saa$,  with the choices $\btha \eq \Ihp (\psia\bth)$ and $\gha \eq \Pihp (\psia \nabla{\cdot} \bv +\nabla \psia {\cdot} \bth) \in\Pp(\Ta)$ to obtain
\begin{equation} \label{eq:etape}
\norm{\saa-\Ihp(\psia\bth)}_{\oma} \leq C \sup_{\substack{\varphi\in H^1_*(\oma)\\\norm{\nabla \varphi}_{\oma}=1}}\left\{ (\gha,\varphi)_{\oma} + (\Ihp(\psia\bth),\nabla \varphi)_{\oma}\right\},
\end{equation}
where the space $H^1_*(\oma)$ is defined in~\eqref{eq:H1s}.
Let $h_{\oma}$ denote the diameter of $\oma$ and recall the Poincar\'e inequality $ \norm{v}_{\oma} \leq C h_{\oma}\norm{\nabla v}_{\oma}$ on $H^1_*(\oma)$, with a constant $C$ depending only on the dimension $\dim$ and on $\kappa_\calT$. Moreover, note that the shape-regularity of the mesh implies that $h_{\oma} \approx h_K \approx h_F$ for all $K\in\Ta$ and all $F \in \calFaint$.

Define for any $\bvh \in \bC^1(\calT)$ the jump $\jump{\bvh}$ on an interior face $F$ shared by two mesh elements $K_+$ and $K_-$ by $\jump{\bvh} \coloneqq (\bvh|_{K_+})|_F - (\bvh|_{K_-})|_F$; here $\bn_F \eq \bn_{K_-}=-\bn_{K_+}$ is the unit normal to $F$ that points outward $K_-$ and inward $K_+$. Similarly, if $F$ is a boundary face, then we define $\jump{\bvh} \coloneqq \bvh|_F$.
To bound the right-hand side of~\eqref{eq:etape}, consider an arbitrary $\varphi\in H^1_*(\oma)$ such that $\norm{\nabla \varphi}_{\oma}=1$. Then, using integration by parts elementwise, we find that
\[
\begin{split}
 \big(\Ihp(\psia\bth), \nabla \varphi\big)_{\oma}
=&\sum_{F\in\calFaint}\big(\jump{\Ihp(\psia\bth)}{\cdot}\bn_F, \varphi\big)_{F} -\sum_{K\in\Ta}\big(\nabla{\cdot} \Ihp(\psi_a\bth), \varphi\big)_K \\
=&\sum_{F\in\calFaint}\big(\Pi_F^p(\psia\jump{\bth}{\cdot}\bn_F), \varphi\big)_{F} -\big(\Pihp(\nabla_{\calT}{\cdot} (\psi_a\bth)), \varphi\big)_{\oma}.
\end{split}
\]
Here, in the first identity, the set of faces can be restricted to $\calFaint$; indeed, for interior vertices, this follows from the fact that $\psia$ vanishes on $\partial\oma$, whereas for boundary vertices,  $\varphi \in H^1_*(\oma)$ vanishes on $\GDa$.
The second identity is then obtained from the definition of the elementwise canonical interpolant $\Ihp$ in~\eqref{eq:Ihp_def} and the commutation identity~\eqref{eq:Ihp_identity}. Expanding $\nabla_{\calT}{\cdot} (\psi_a\bth) = \nabla\psia{\cdot}\bth + \psia\nabla_{\calT}{\cdot}\bth$ and simplifying gives
\begin{equation}\label{eq:local_res_identity} \begin{split}
(\gha,\varphi)_{\oma} + (\Ihp(\psia\bth),\nabla\varphi)_{\oma}
= {} & (\Pihp(\psia \nabla_{\calT}{\cdot}(\bv-\bth)),\varphi)_{\oma}\\ {} & + \sum_{F\in\calFaint}\big(\Pi_F^p(\psia\jump{\bth}{\cdot}\bn_F), \varphi\big)_{F}.
\end{split}
\end{equation}
We now bound the two terms on the right-hand side of~\eqref{eq:local_res_identity} separately.

To bound the first term, we consider first the case $p\geq 1$: using the divergence constraint on $\bth$ in~\eqref{eq_L2_proj_constraint}, the orthogonality of the $L^2$-projections, the approximation bound $\norm{\varphi-\Pi_{\calT}^{p-1}\varphi }_K \leq C \frac{h_K}{p+1}\norm{\nabla\varphi}_K$ (note that $\frac1p\le \frac{2}{p+1}$ for all $p\ge1$), along with $\norm{\psia}_{\infty,\oma} = 1$ and $\norm{\nabla \varphi }_{\oma}=1$, we find that there is a constant $C$, depending only on $\dim$ and $\kappa_\calT$, such that
\begin{align*}
(\Pihp(\psia \nabla_{\calT}{\cdot}(\bv-\bth)),\varphi)_{\oma} = {} & \left(\nabla{\cdot}\bv - \Pihp(\nabla{\cdot}\bv), \psia \Pihp (\varphi - \Pi_{\calT}^{p-1} \varphi) \right)_\oma \\ \leq {} & C \left\{\sum_{K\in\Ta} \frac{h_K^2}{(p+1)^2} \norm{\nabla{\cdot}\bv - \Pihp(\nabla{\cdot}\bv)}_K^2\right\}^{\frac{1}{2}} \ifIMA \else\\ \fi \leq C \left\{\sum_{K\in\Ta} [\Eloc(\bv)]^2\right\}^{\frac{1}{2}}.
\end{align*}
For $p=0$, we instead apply the Cauchy--Schwarz inequality, the stability of the $L^2$-projection, the Poincar\'e inequality on $H^1_*(\oma)$, and $\norm{\nabla \varphi }_{\oma}=1$ to get
\begin{align*}
\abs{(\Pihp(\psia \nabla_{\calT}{\cdot}(\bv-\bth)),\varphi)_{\oma}} & \leq C \norm{\nabla_{\calT}{\cdot}(\bv - \bth)}_\oma h_{\oma}\norm{\nabla \varphi}_{\oma}\\ & \leq C \left\{\sum_{K\in\Ta} [\Eloc(\bv)]^2\right\}^{\frac{1}{2}},
\end{align*}
where $C$ depends only on $\dim$ and $\kappa_\calT$.

To bound the second term on the right-hand side of~\eqref{eq:local_res_identity}, we recall the trace inequality
\[
\norm{\varphi}^2_F \leq C \left( \norm{\nabla\varphi}_K\norm{\varphi}_K + h_K^{-1} \norm{\varphi}_K^2\right),
\]
for any $\varphi \in H^1(K)$ and $F\in\calF_K$, where $C$ depends only on $\dim$ and $\kappa_\calT$. Combined with the Poincar\'e inequality on $H^1_*(\oma)$ and $\norm{\nabla\varphi}_\oma=1$, this gives
\[
\sum_{F\in\calFaint}\abs{\big(\Pi_F^p(\psia\jump{\bth}{\cdot}\bn_F), \varphi\big)_{F}} \leq C\left\{\sum_{F\in\calFaint} h_F\norm{\jump{\bth}{\cdot}\bn_F}_F^2 \right\}^{\frac{1}{2}},
\]
with $C$ depending only on $\dim$ and $\kappa_\calT$.
Finally, we invoke Lemma~\ref{lem:faces}, yielding, for each $F\in\calFaint$,
\begin{equation}\label{eq:bound_term_2_p_dependence}
    h_F^{1/2}\norm{\jump{\bth}{\cdot}\bn_F}_F \leq C \sup_{\substack{w\in H^1(\omf)\\ w=0 \text{ on } \pt \omf \setminus F\\ \norm{\nabla w}_{\omf}=1}} (\jump{\bth}{\cdot}\bn_F,w)_F,
\end{equation}
where now the constant $C$ depends on the polynomial degree $p$ in addition to $\dim$ and $\kappa_\calT$.
Fix $w\in H^1(\omf)$ such that $w=0$ on $\pt \omf \setminus F$ and $\norm{\nabla w}_{\omf}=1$. By definition, $F\in\calFaint$ means that $F$ is either an internal face shared by two simplices, or a Neumann boundary face. Then, the zero extension of $w$ to $\Omega$ belongs to $H^1_{\GD}(\Omega)$. Since $\bv \in \Hdivzero$, we infer from the definition of the weak divergence that
\[
    (\nabla{\cdot}\bv, w)_{\omf} + (\bv,\nabla w)_{\omf} = 0.
\]
Consequently, developing $(\jump{\bth}{\cdot}\bn_F,w)_F$ shows that
\begin{align*}
  \abs{  (\jump{\bth}{\cdot}\bn_F,w)_F } & =  \abs{(\nabla_{\calT}{\cdot}\bth, w)_{\omf} + (\bth,\nabla w)_{\omf}} \\
  & \leq \abs{ (\nabla_{\calT}{\cdot}(\bth-\bv), w-\Pihp w)_{\omf}} + \abs{ (\bth - \bv,\nabla w)_{\omf}}\\
    & \leq \norm{\nabla_{\calT}{\cdot}(\bth-\bv)}_{\omf} \norm{w -\Pihp w}_{\omf} + \norm{\bth - \bv}_{\omf} \norm{\nabla w}_{\omf} \\
    & \leq C \sum_{K\in\TF} \left\{\norm{\bv - \bth}_K^2  + \frac{h_K^2}{(p+1)^2} \norm{\nabla{\cdot}(\bv - \bth)}_K^2\right\}^{1/2},
\end{align*}
owing to the Cauchy--Schwarz inequality, the orthogonality of the $L^2$-projection, and the Poincar\'e--Friedrichs inequality $\norm{w}_{\omf} \leq C h_F \norm{\nabla w}_{\omf}$. Hence, Lemma~\ref{lem:local_const_equivalence} below implies that
\[
\sum_{F\in\calFaint}\big(\Pi_F^p(\psia\jump{\bth}{\cdot}\bn_F), \varphi\big)_{F} \leq C \left\{ \sum_{K\in \Ta} [\Eloc(\bv)]^2  \right\}^{\frac{1}{2}},
\]
where the constant $C$ depends only on $\dim$, $\kappa_\calT$, and the polynomial degree $p$ via~\eqref{eq:bound_term_2_p_dependence}. Combining these bounds implies~\eqref{eq:sha_bound}.
\end{proof}

\paragraph{Step~2. Bound on $\norm{\bv-P^p_{\calT} (\bv)}_{K}$.}

Let $K \in \calT$. In this second and last step, we first show that
\begin{equation} \label{eq:eff}
    \norm{P^p_{\calT} (\bv) - \bth}_{K} \leq C \left\{ \sum_{K^{'\!\!}\in\calT_{K}} [\Eloct(\bv)]^2 \right\}^{\frac{1}{2}}.
\end{equation}
Recalling that $\calV_K$ denotes the set of vertices of the element $K$, using the partition of unity $\sum_{\ver\in\calV_{K}}\psia|_K = 1$ and the linearity of the elementwise canonical interpolant $\Ihp$~\eqref{eq:Ihp_def} as well as definition~\eqref{eq:bsh_def} of $P^p_{\calT} (\bv)$ and the fact that $\bth=\Ihp(\bth)$, we find that
\begin{align*}
 \left(P^p_{\calT}  (\bv) - \bth\right)|_K = \left(P^p_{\calT} (\bv) -  \Ihp(\bth)\right)|_K  = \sum_{\ver\in\calV_{K}}\left(\saa-\Ihp(\psia\bth) \right)|_K.
\end{align*}
Thus,
\[
    \norm{P^p_{\calT} (\bv) -\bth}_K^2 = \Bnorm{\sum_{\ver\in\VK}\big(\saa - \Ihp(\psia\bth)\big)}_K^2 \leq (d+1) \sum_{\ver \in \calV_{K}}\norm{\saa-\Ihp(\psia\bth)}_\oma^2,
\]
and Lemma~\ref{lem:bound_sha} then yields~\eqref{eq:eff}.

Finally, having obtained~\eqref{eq:eff}, the main bound~\eqref{eq:proof_main_bound} then follows from the triangle inequality and  Lemma~\ref{lem:local_const_equivalence}, since
\[
\norm{\bv - P^p_{\calT} (\bv)}_{K} \leq \norm{\bv-\bth}_{K} + \norm{\bth - P^p_{\calT} (\bv)}_{K} \leq C \left\{ \sum_{K^{'\!\!}\in\calT_{K}} [\Eloct(\bv)]^2 \right\}^{\frac{1}{2}}.
\]
This completes the proof of the approximation property~\eqref{eq:proj_approx} from Theorem~\ref{thm:proj}.

\subsection{Proof of the projection property~\eqref{eq:proj},
$\bL^2$ stability~\eqref{eq:stab}, and $\Hdiv$ stability~\eqref{eq:stab_div}}\label{sec:proof_proj_concl}

To prove~\eqref{eq:proj}, we observe that if $\bv \in \RTN \cap \Hdivzero$, then it follows from the definition~\eqref{eq:local_minimizers} that $\Eloc(\bm{v}) = 0$ for all $K \in \calT$, and thus~\eqref{eq:proj} follows immediately from~\eqref{eq:proj_approx}.

To prove~\eqref{eq:stab}, we observe that, for any $K \in \calT$, the triangle inequality yields
\[
    \norm{P^p_{\calT} (\bv)}_K \leq \norm{\bv}_K + \norm{\bv-P^p_{\calT} (\bv)}_K.
\]
The first term is trivially contained in the right-hand side of~\eqref{eq:stab}. Bounding the second one by~\eqref{eq:proj_approx}, the definition~\eqref{eq:local_minimizers} of $\Eloc(\bm{v})$ implies that
\[
    \Eloc(\bm{v}) \leq \norm{\bv}_K + \frac{h_K}{(p+1)}\norm{\nabla{\cdot}\bv - \Pihp(\nabla{\cdot}\bv)}_K.
\]
This shows that~\eqref{eq:stab} holds true.

Finally, from~\eqref{eq:stab}, the bound in~\eqref{eq:stab_div} follows immediately since $\frac{h_K}{p+1} \leq h_\Omega$ and since both terms $\norm{\Pihp (\nabla{\cdot} \bv)}_K$ and $\norm{\nabla{\cdot} \bv - \Pihp (\nabla{\cdot} \bv)}_K$ are bounded by $\norm{\nabla{\cdot} \bv}_K$.

\section{Proof of Theorem~\ref{thm:appr_est} ($hp$-optimal approximation estimates under minimal regularity)} \label{sec:proof_hp}

We present here a proof of Theorem~\ref{thm:appr_est}. For this purpose, we will combine Theorem~\ref{thm:loc_glob} with its unbalanced but polynomial-degree-robust variant that we develop first.

\subsection{Polynomial-degree-robust one-sided bound}\label{sec:equiv_p_rob}

We present her an auxiliary result which gives a bound where the global-best approximation error~\eqref{eq:global_local_equiv_1} is bounded in terms of the sums of local-best approximation errors~\eqref{eq:local_minimizers} with a constant that is {\em robust with respect to the polynomial degree}, but where the polynomial degree in the local approximation errors is $(p-1)$ instead of $p$. As a result, in contrast to Theorem~\ref{thm:loc_glob}, this is a one-sided inequality and not an equivalence, and it is valid only for $p\geq 1$.

\begin{proposition}[Polynomial-degree-robust bound]\label{prop:loc_glob_p+1}
There exists a  constant~$C$, depending only on the space dimension $\dim$ and the shape-regularity parameter $\kappa_\calT$ of $\calT$, such that, for any $\bv\in \Hdivzero$ and any $p\geq 1$,
\begin{equation} \label{eq:LocalGlobal:p+1}
[\Eglob(\bv)]^2 \leq C \sum_{K\in \calT} [\Elocp(\bv)]^2.
\end{equation}
\end{proposition}

The proof of Proposition~\ref{prop:loc_glob_p+1} is done in the same spirit as that of Theorem~\ref{thm:loc_glob}.
Let $\bv\in \Hdivzero$.
In order to show~\eqref{eq:LocalGlobal:p+1}, it again is enough to find $\bsh\in \RTN\cap \Hdivzero$ such that $\nabla{\cdot}\bsh=\Pihp(\nabla{\cdot}\bv)$ and
\begin{equation} \label{eq:bsh_bound_simplified_p+1}
\norm{\bv-\bsh} \leq C \left\{\sum_{K\in \calT} [\Elocp(\bv)]^2\right\}^{\frac{1}{2}},
\end{equation}
where $C$ is a  constant depending only on $\dim$ and $\kappa_\calT$.
To this purpose, we adapt Definition~\ref{def:proj} as follows.

\begin{definition}[Alternative locally-defined mapping from $\Hdivzero$ to $\RTN\cap \Hdivzero$] \label{def:proj_bis}
Let $\bv\in \Hdivzero$ be arbitrary. Let $\bth$ be defined elementwise by
\begin{equation}\label{eq_L2_proj_-1}
    \bth|_K \eq \argmin_{\substack{\bv_K \in \RTNKm \\ \nabla{\cdot}\bv_K=\Pi_{\calT}^{p-1}(\nabla{\cdot}\bv) }} \norm{\bv-\bv_K}_K \qquad \forall K \in \calT.
\end{equation}
For each mesh vertex $\ver\in \calV$, the patchwise contributions $\saa$ are now defined as
\begin{equation}\label{eq:sha_def_-1}
\saa \coloneqq \argmin_{\substack{\bva \in \Vha \\ \nabla{\cdot}\bva = \Pihp (\psia \nabla{\cdot} \bv) + \nabla \psia {\cdot} \bth}}
\norm{\bva - \psia\bth}_{\oma},
\end{equation}
with the spaces $\Vha$ still defined in~\eqref{eq:Vha_def}. Finally, after extending each $\saa$ from $\oma$ to the rest of $\Omega$ by zero, the equilibrated flux reconstruction $\bsh \in \RTN\cap \Hdivzero$, is defined as
\begin{equation}\label{eq:bsh_def_-1}
    \bsh \eq \sum_{\ver\in\calV}\saa.
\end{equation}
\end{definition}

Note that the elementwise minimization in~\eqref{eq_L2_proj_-1} is done over $(p-1)$-degree Raviart--Thomas--N\'{e}d\'{e}lec spaces, in contrast to~\eqref{eq_L2_proj_constraint}, and the RTN interpolation $\Ihp$ is not used in~\eqref{eq:sha_def_-1}, in contrast to~\eqref{eq:sha_def}.

Since the orthogonality property~\eqref{eq:orth_K} also holds here, we infer that~\eqref{eq:rtn_orthogonality} still holds with the above definitions. This in turn gives the necessary compatibility condition yielding existence and uniqueness for the local minimization problems~\eqref{eq:sha_def_-1} in the spirit of Lemma~\ref{lem:ex_un}.
Finally, just as in~\eqref{eq:equilibrated_proof}, we deduce that $\nabla{\cdot}\bsh=\Pihp(\nabla{\cdot}\bv)$.
It thus remains to prove that
\begin{equation} \label{eq:eff_p+1}
    \norm{\bth - \bsh} \leq C \left\{\sum_{K\in \calT} [\Elocp(\bv)]^2\right\}^{\frac{1}{2}},
\end{equation}
with $C$ only depending on $\dim$ and $\kappa_\calT$. Then~\eqref{eq:bsh_bound_simplified_p+1} follows from~\eqref{eq:eff_p+1} by the triangle inequality $\norm{\bv-\bsh} \leq \norm{\bv-\bth} + \norm{\bth-\bsh}$, where the divergence-constrained minimization in $\norm{\bv-\bth}_K$ is subordinate to the unconstrained one in $\Elocp(\bv)$ by Lemma~\ref{lem:local_const_equivalence} below applied with $(p-1)$ in place of $p$.

\begin{lemma}[Bound on $\saa$]\label{lem:bound_sha_p+1}
There exists a constant $C$, depending only on $\dim$ and $\kappa_\calT$, such that
\begin{equation}\label{eq:sha_bound_p+1}
\norm{\saa- \psia\bth}_{\oma}\leq C  \left\{\sum_{K\in\Ta} [\Elocp(\bv)]^2\right\}^{\frac{1}{2}} \qquad \forall \ver \in \calV.
\end{equation}
\end{lemma}
\begin{proof}
Fix a vertex $\ver \in \calV$. We rely on Lemma~\ref{lem:main_stability_bound}, where we take $\btha \eq \psia\bth$ and $\gha \eq \Pihp (\psia \nabla{\cdot} \bv) + \nabla \psia {\cdot} \bth$ in order to apply it to our construction~\eqref{eq:sha_def_-1} from Definition~\ref{def:proj_bis}. This yields
\[
    \norm{\saa- \psia\bth}_{\oma} \leq C \sup_{\substack{v\in H^1_*(\oma)\\ \norm{\nabla \varphi}_{\oma}=1}} \left\{(\gha,\varphi)_{\oma} + (\btha,\nabla \varphi)_{\oma} \right\}.
\]
Let $\varphi\in H^1_*(\oma)$ with $\norm{\nabla \varphi}_{\oma}=1$ be fixed, where we recall that the space $H^1_*(\oma)$ is defined in~\eqref{eq:H1s}. Then, the product $\psia\varphi \in H^1_{\GD}(\Omega)$ for any $\ver\in\calV$ and thus the definition of the weak divergence implies that
\[
    \big(\bv, \nabla (\psia \varphi)\big)_{\oma} + \big(\nabla{\cdot} \bv, \psia \varphi \big)_{\oma} = 0.
\]
Then, the product rule and the orthogonality of the $L^2$-projection give
\begin{align*}
(\gha,\varphi)_{\oma} + (\btha,\nabla \varphi)_{\oma}
= {} & \big(\Pihp(\psi_a\nabla{\cdot} \bv), \varphi\big)_\oma  +(\nabla\psia{\cdot} \bth,\varphi)_{\oma} + \big(\psia\bth, \nabla \varphi\big)_{\oma} \\
= {} & \big(\nabla{\cdot} \bv,\psia \Pihp(\varphi)\big)_{\oma} + \big(\bth, \nabla (\psia \varphi)\big)_{\oma}\\
= {} & \big(\nabla{\cdot} \bv, \psia(\Pihp(\varphi)-\varphi)\big)_{\oma} + \big(\bth-\bv, \nabla (\psia \varphi)\big)_{\oma}\\
= {} & \big(\psia(\nabla{\cdot} \bv-\Pi_{\calT}^{p-1}(\nabla{\cdot}\bv)),  \Pihp(\varphi)-\varphi \big)_{\oma} + \big(\bth-\bv, \nabla (\psia \varphi)\big)_{\oma},
\end{align*}
since $\psia \Pi_{\calT}^{p-1}(\nabla{\cdot}\bv)$ is a piecewise polynomial of degree at most $p$.
Therefore, we have
\begin{align*}
\abs{(\gha,\varphi)_{\oma} + (\btha,\nabla \varphi)_{\oma}}
\leq {}& C  \sum_{K\in\Ta} \Big[\frac{h_K}{p}\norm{\nabla{\cdot} \bv-\Pi_{\calT}^{p-1}(\nabla{\cdot} \bv)}_K\Big] \norm{\nabla \varphi}_{\oma}\\
{} & +  \|\bv-\bth\|_{\oma}\norm{\nabla(\psia \varphi)}_{\oma}\\
\leq {} &C \left(1+\norm{\nabla(\psia \varphi)}_{\oma}\right) \left\{\sum_{K\in \Ta} [\Elocp(\bv)]^2\right\}^{\frac{1}{2}},
\end{align*}
where we have used $\norm{\psia}_{\infty, \oma}=1$, the $hp$ approximation bound $\norm{\varphi-\Pi_{\calT}^{p}(\varphi) }_K \leq C \frac{h_K}{p+1}\norm{\nabla\varphi}_K \leq C \frac{h_K}{p}\norm{\nabla\varphi}_K$, the Cauchy--Schwarz inequality, the scaling $\norm{\nabla\varphi}_{\oma}=1$, and Lemma~\ref{lem:local_const_equivalence}.
Finally, the bound~\eqref{eq:sha_bound_p+1} follows from the inequality $\norm{\nabla(\psia \varphi)}_{\oma}  \leq  C \norm{\nabla \varphi}_\oma \leq C$ for all $\varphi\in H^1_*(\oma)$, owing to the Poincar\'e inequality on $H^1_*(\oma)$ and $\norm{\nabla \varphi}_{\oma}=1$.
\end{proof}

Finally, we obtain~\eqref{eq:eff_p+1} from Lemma~\ref{lem:bound_sha_p+1} and the estimate
\[
    \norm{\bsh-\bth}^2 = \sum_{K \in \calT}\Bnorm{\sum_{\ver\in\VK}\big(\saa - \psia\bth\big)}_K^2 \leq (d+1) \sum_{\ver \in \calV}\norm{\saa-\psia\bth}_\oma^2.
\]
As explained above, \eqref{eq:eff_p+1} then implies~\eqref{eq:LocalGlobal:p+1} and completes the proof of Proposition~\ref{prop:loc_glob_p+1}.

\subsection{Proof of Theorem~\ref{thm:appr_est}}

The proof of Theorem~\ref{thm:appr_est} hinges on the bounds from Theorem~\ref{thm:loc_glob} and Proposition~\ref{prop:loc_glob_p+1}.
Recall the definitions~\eqref{eq:global_local_equiv_1} of $\Eglob(\bm{v})$ and~\eqref{eq:local_minimizers} of $\Eloc(\bm{v})$.
Recall also the notation $\delta_{s<1}\coloneqq1$ if $s<1$ and $\delta_{s<1}\coloneqq0$ if $s\ge1$.
We proceed in two steps.

\paragraph{Step~1. Case $p\leq s$.} We first suppose that $p\leq s$ and let $t \eq \min(s,p+1)$.
Here, we will employ Theorem~\ref{thm:loc_glob}.
Since $\Pp(K;\R^\dim) \subset \RTNK$, well-known $hp$-approximation bounds, see e.g. \cite[Lemma~4.1]{Bab_Sur_hp_FE_87}, imply that
\begin{equation}\label{eq:local_error_rates}
[ \Eloc(\bv)]^2 \leq C \bigg\{
\Big[ \frac{h_K^{t}}{(p+1)^{s}} \norm{v}_{\bH^s(K)}\Big]^2 + \delta_{s<1}\Big[\frac{h_K}{p+1} \norm{\nabla{\cdot}\bv}_K\Big]^2\bigg\},
\end{equation}
for each $K\in \calT$, with $C$ depending only on $s$, $\dim$, $\kappa_\calT$. Note that for $s<1$, we applied here the trivial bound $\norm{\nabla{\cdot}\bv - \Pihp(\nabla{\cdot}\bv)}_K\leq \norm{\nabla{\cdot}\bv}_K$ as $\bv|_K \in \bH^s(K)$ is insufficient to improve the bound on the error of the divergence. Combining~\eqref{eq:local_error_rates} with the first bound in~\eqref{eq:global_local_equiv_2} of Theorem~\ref{thm:loc_glob} then implies that there exists a constant  $C_{s,\dim,\kappa_\calT,p}$ depending only on $s$, $\dim$, $\kappa_\calT$, and $p$, such that
\[
[ \Eglob(\bv)]^2 \leq C_{s,\dim,\kappa_\calT,p}
\sum_{K\in\calT} \bigg\{\Big[ \frac{h_K^{t}}{(p+1)^{s}} \norm{v}_{\bH^s(K)}\Big]^2 + \delta_{s<1}\Big[\frac{h_K}{p+1} \norm{\nabla{\cdot}\bv}_K\Big]^2 \bigg\}.
\]
Define then the constant $C_{s,\dim,\kappa_\calT}^{\star} \eq \max_{0\leq p \leq s} C_{s,\dim,\kappa_\calT,p}$, so that, for all $p\leq s$,
\[
[ \Eglob(\bv)]^2 \leq  C_{s,\dim,\kappa_\calT}^{\star} \sum_{K\in\calT} \bigg\{\Big[ \frac{h_K^{t}}{(p+1)^{s}} \norm{v}_{\bH^s(K)}\Big]^2 + \delta_{s<1}\Big[\frac{h_K}{p+1} \norm{\nabla{\cdot}\bv}_K\Big]^2 \bigg\}.
\]
This implies~\eqref{eq:est_p+1} for any $p \leq s$ with constant $C=C_{s,\dim,\kappa_\calT}^{\star}$.

\paragraph{Step~2. Case $p>s$.} Now consider the case $p>s$; since $p$ is an integer, this implies that $p\geq 1$. Here we rely on Proposition~\ref{prop:loc_glob_p+1}.
The approximation bounds, similarly to in~\eqref{eq:local_error_rates}, imply that there exists a constant $C$, depending only on $s$, $\dim$, and $\kappa_\calT$, such that
\[
[\Elocp(\bv)]^2 \leq C \bigg\{
\Big[ \frac{h_K^{s}}{p^{s}}\norm{\bv}_{\bH^{s}(K)}\Big]^2 + \delta_{s<1}\Big[\frac{h_K}{p} \norm{\nabla{\cdot}\bv}_K\Big]^2 \bigg\},
\]
for all $K\in\calT$.
Note that $p+1\leq 2p$ for all $p\geq 1$, so that the terms $p^{s}$ in the denominators above can be replaced by $(p+1)^{s}$ at the cost of an extra $s$-dependent constant, and similarly for $1/p \leq 2/(p+1)$.
Hence, the inequality~\eqref{eq:LocalGlobal:p+1} of Proposition~\ref{prop:loc_glob_p+1} and summation over the elements of~$\calT$ shows that there exists a constant $C_{s,\dim,\kappa_{\calT}}^{\sharp}$  depending only on $s$, $\dim$, and $\kappa_\calT$ such that~\eqref{eq:est_p+1} holds with constant $C=C_{s,\dim,\kappa_{\calT}}^{\sharp}$ for all $p>s$.

\paragraph{Conclusion.} Combining Steps 1 and 2 shows that \eqref{eq:est_p+1} holds for general $s\geq 0$ and $p\geq 0$ with a constant $C$ that can be taken as $\max\{C_{s,\dim,\kappa_\calT}^{\star},C_{s,\dim,\kappa_\calT}^{\sharp}\}$, which then depends only on $s$, $\dim$, and $\kappa_\calT$.

\begin{remark}[Full $hp$-optimality]
Theorem~\ref{thm:appr_est} shows that optimal order convergence rates with respect to both the mesh-sizes $h_K$ and the polynomial degree $p$ can be obtained despite the unfavorable dependence of the constant $C$ on the polynomial degree $p$ in Theorem~\ref{thm:loc_glob} and unbalanced polynomial degrees in Proposition~\ref{prop:loc_glob_p+1}.
\end{remark}

\section{Application to a priori error estimates}\label{sec:MFEs_a_priori}

In this section we show how to apply the results of Section~\ref{sec:Meain-results-Statements} to the a priori error analysis of mixed finite element methods and least-squares mixed finite element methods for a model diffusion problem.

\subsection{Mixed finite element methods}
Let us consider the dual mixed finite element method for the Poisson model problem, following \ifIMA\else Raviart and Thomas~\fi\cite{Ra_Tho_MFE_77}, \ifIMA\else N{\'e}d{\'e}lec~\fi\cite{Ned_mix_R_3_80}, \ifIMA\else Roberts and Thomas~\fi\cite{Ro_Tho_91}, or \ifIMA\else Boffi~\eal~\fi\cite{Bof_Brez_For_MFEs_13}.
Let $f\in L^2(\Omega)$ and $\GN = \emptyset$ for simplicity, so that $\Hdivzero$ becomes $\HdivO$. Consider the Laplace problem of finding $u: \Omega \rightarrow \R$ such that
\begin{subequations}\label{eq:Lapl}\begin{alignat}{2}
    - \Delta u & = f \qquad & & \text{in } \Omega, \\
    u & = 0  & & \text{on } \pt \Omega.
\end{alignat}\end{subequations}
The primal weak formulation of~\eqref{eq:Lapl} reads: find $u\in \Hunz$ such that
\begin{equation}\label{eq:Lap:prim}
  (\nabla u, \nabla v) = (f,v) \qquad \forall v \in \Hunz.
\end{equation}
The dual weak formulation of~\eqref{eq:Lapl} then reads: find $\bsig \in \HdivO$ and $u\in L^2(\Omega)$ such that
\begin{subequations}\label{eq:MF}
\begin{alignat}{2}
(\bsig,\bv)-(u,\nabla{\cdot} \bv)&=0 \qquad \qquad & & \forall \bv \in \HdivO,\label{eq:MF1}\\
(\nabla{\cdot}\bsig,q)&=(f,q) & & \forall q\in L^2(\Omega). \label{eq:MF2}
\end{alignat}
Classically, $u$ from~\eqref{eq:Lap:prim} and~\eqref{eq:MF} coincide and $\bsig = - \nabla u$.
\end{subequations}
The dual mixed finite element method of order $p \geq 0$ for the problem~\eqref{eq:MF} then looks for the pair $\bsig_{\mathrm{M}}\in \RTN\cap \HdivO$ and $u_{\mathrm{M}}\in \Pp(\calT)$ such that
\begin{subequations}\label{eq:DMF}
\begin{alignat}{2}
(\bsig_{\mathrm{M}},\bvh)-(u_{\mathrm{M}},\nabla{\cdot} \bvh)&=0 \qquad \qquad & & \forall \bvh \in \RTN\cap \HdivO,\label{eq:DMF1}\\
(\nabla{\cdot}\bsig_{\mathrm{M}},q_{\calT})&=(f,q_{\calT}) & & \forall q_{\calT}\in \Pp(\calT).\label{eq:DMF2}
\end{alignat}
\end{subequations}
It is immediate to check from~\eqref{eq:MF2} and~\eqref{eq:DMF2} that $\nabla {\cdot} \bsig_{\mathrm{M}}=\Pihp(\nabla{\cdot} \bsig)$.  Furthermore, the following a priori error characterization is classical, \cf~\cite{Bof_Brez_For_MFEs_13}. We include its proof to highlight the precise arguments.

\begin{lemma}[A priori bound for mixed finite element methods] \label{lem:a_pr_DM} Let $\bsig_{\mathrm{M}}$ be the first component of the dual mixed finite solution solving~\eqref{eq:DMF}, approximating $\bsig$ from~\eqref{eq:MF}. Then
\begin{align*}
\norm{\bsig-\bsig_{\mathrm{M}}} = \min_{\substack{\bvh\in \RTN\cap \HdivO\\ \nabla{\cdot}\bvh=\Pihp(\nabla{\cdot}\bsig)}} \norm{\bsig-\bvh}.
\end{align*}
\end{lemma}

\begin{proof}
Subtracting~\eqref{eq:DMF1} from~\eqref{eq:MF1}, we have
\begin{equation} \label{eq:MFE_orth}
(\bsig-\bsig_{\mathrm{M}},\bvh)-(u-u_{\mathrm{M}},\nabla{\cdot} \bvh) =0\quad \forall \bvh \in \RTN\cap \HdivO.
\end{equation}
Let $\bsh\in \RTN\cap \HdivO$ be such that $\nabla{\cdot}\bsh=\Pihp(\nabla{\cdot}\bsig)$. Taking $\bvh=\bsh-\bsig_{\mathrm{M}}$ in~\eqref{eq:MFE_orth}, we obtain, since $\nabla{\cdot}\bvh=0$,
\begin{align*}
(\bsig-\bsig_{\mathrm{M}},\bsh-\bsig_{\mathrm{M}}) =0.
\end{align*}
Now clearly
\begin{align*}
\norm{\bsig-\bsig_{\mathrm{M}}}^2=(\bsig-\bsig_{\mathrm{M}},\bsig-\bsig_{\mathrm{M}}) =(\bsig-\bsig_{\mathrm{M}},\bsig-\bsh) \leq \norm{\bsig-\bsig_{\mathrm{M}}}\,\norm{\bsig-\bsh},
\end{align*}
and hence $\norm{\bsig-\bsig_{\mathrm{M}}}\leq \norm{\bsig-\bsh}$. Since $\bsh$ is arbitrary subject to the divergence constraint and can be taken as $\bsig_{\mathrm{M}}$, we obtain the assertion.\end{proof}

Thus, $\norm{\bsig-\bsig_{\mathrm{M}}}$ can be readily estimated by using Theorems~\ref{thm:loc_glob} and~\ref{thm:appr_est}.
\subsection{Least-squares mixed finite element methods}\label{sec:LS_a_priori}
In this subsection, we showcase the application of our results to the least-squares mixed finite element method discussed in \ifIMA\else Pehlivanov~\eal~\fi\cite{Peh_Car_Laz_LS_MFEs_94}, \ifIMA\else Cai and Ku~\fi\cite{Cai_Ku_LS_MFE_L2_10}, and \ifIMA\else Ku~\fi\cite{Ku_LS_MFE_min_reg_13}, see also the references therein.

Let again $\GN = \emptyset$ for simplicity and $f\in L^2(\Omega)$. Let $\bsig\in\HdivO$ and $u\in \Hunz$ be such that
\begin{align*}
(\bsig,u)\coloneqq\arg\min_{(\bp, v)\in\HdivO\times \Hunz}\left\{\LO^2\norm{\nabla{\cdot}\bp-f}^2+\norm{\bp+\nabla v}^2\right\},
\end{align*}
where we recall that $\LO$ is a length scale equal to the diameter of $\Omega$.
Then $\bsig\in\HdivO$ and $u\in \Hunz$ solve the following system of equations:
\begin{subequations}\label{eq:LSMF}
\begin{alignat}{2}
(\bsig+\nabla u,\nabla v)&=0 \qquad \qquad \qquad & & \forall v \in \Hunz,\label{eq:LSMF1}\\
\LO^2(\nabla{\cdot}\bsig,\nabla{\cdot}\bp)+(\bsig+\nabla u,\bp)&=\LO^2(f,\nabla{\cdot}\bp) & & \forall \bp\in \HdivO. \label{eq:LSMF2}
\end{alignat}
\end{subequations}
Again, $\bsig$ and $u$ coincide with the solutions of~\eqref{eq:Lap:prim} and~\eqref{eq:MF}.
Let $p\geq 0$ and $q\geq 1$ denote two fixed polynomial degrees. The least-squares mixed finite element method for the problem~\eqref{eq:LSMF} consists of finding
$\bsig_{\mathrm{LS}} \in \RTN\cap\HdivO$ and $u_{\mathrm{LS}}\in \calP^q(\calT)\cap \Hunz$ such that
\begin{subequations}\label{eq:DLSMF}
\begin{alignat}{2}
(\bsig_{\mathrm{LS}}+\nabla u_{\mathrm{LS}},\nabla v_{\calT})&=0 \,\qquad \qquad \qquad & & \forall v_{\calT} \! \in \! \calP^q(\calT) \! \cap \! \Hunz,\label{eq:DLSMF1}\\
\LO^2(\nabla{\cdot}\bsig_{\mathrm{LS}},\nabla{\cdot}\bp_{\calT})+(\bsig_{\mathrm{LS}}+\nabla u_{\mathrm{LS}},\bp_{\calT})&=\LO^2(f,\nabla{\cdot}\bp_{\calT}) & & \forall \bp_{\calT} \! \in \! \RTN \! \cap \! \HdivO. \label{eq:DLSMF2}
\end{alignat}
\end{subequations}
Similarly to Lemma~\ref{lem:a_pr_DM}, we can obtain the following a priori error characterization.

\begin{lemma}[A priori bound for least-squares mixed finite element methods] \label{lem:a_pr_LSM} Let ($\bsig_{\mathrm{LS}}, u_{\mathrm{LS}})$ be the least-squares mixed finite solution pair solving~\eqref{eq:DLSMF}, approximating $(\bsig,u)$ from~\eqref{eq:LSMF}. Then there exists a  generic constant $C$, at most equal to $17$, such that
\begin{align*}
\|\bsig-\bsig_{\mathrm{LS}}\| + \|\nabla (u-u_{\mathrm{LS}})\|
\leq  {} & C \left(
\min_{\substack{\bvh\in \RTN\cap \HdivO\\\nabla{\cdot}\bvh=\Pihp(\nabla{\cdot}\bsig)}} \|\bsig-\bvh\| +\min_{v_{\calT}\in \calP^q(\calT)\cap \Hunz}\|\nabla(u-v_{\calT})\|\right).
\end{align*}
\end{lemma}

\begin{proof}
Define the bilinear form $\calA$ on $(\HdivO\times \Hunz)\times (\HdivO\times \Hunz)$ by
\begin{align*}
\calA(\bsig,u;\bp,v)\coloneqq(\bsig+\nabla u,\nabla v)+\LO^2(\nabla{\cdot}\bsig,\nabla{\cdot}\bp)+(\bsig+\nabla u,\bp).
\end{align*}
We have the following orthogonality from~\eqref{eq:LSMF} and~\eqref{eq:DLSMF}:
\begin{align}\label{eq:OrthoGonality:LSMFE}
\calA(\bsig-\bsig_{\mathrm{LS}},u-u_{\mathrm{LS}};\bp_{\calT},v_{\calT})=0
\end{align}
for all $\bp_{\calT}\in \RTN\cap\HdivO$ and for all $v_{\calT}\in \calP^q(\calT)\cap \Hunz$.
Moreover, the following coercivity is known from~\cite{Peh_Car_Laz_LS_MFEs_94}: there exists a  constant $C$ such that
\begin{equation}\label{eq:Coercivity:LSMFE}
\calA(\bp,v;\bp,v) \geq \frac{1}{C} \left(\|\bp\|^2 + \LO^2\|\nabla{\cdot}\bp\|^2+\|\nabla v\|^2\right) \qquad \forall (\bp, v)\in\HdivO\times \Hunz).
\end{equation}
Indeed, owing to the Cauchy--Schwarz and Young inequalities, we have, for any $0 < \varepsilon < 2$,
\begin{align*}
\calA(\bp,v;\bp,v)
= {} &\|\nabla v\|^2 + (2-\varepsilon)(\bp,\nabla v)+ \LO^2 \|\nabla{\cdot}\bp\|^2 + \|\bp\|^2 + \varepsilon(\nabla v,\bp)\\
\geq {} & \|\bp\|^2 + \|\nabla v\|^2 - \frac{2 - \varepsilon}{2}(\|\bp\|^2 + \|\nabla v\|^2) + \LO^2 \|\nabla{\cdot}\bp\|^2 - \varepsilon \|\nabla{\cdot}\bp\|\|v\|\\
\geq {} & \frac{\varepsilon}{2}(\|\bp\|^2 + \|\nabla v\|^2) + \LO^2 \|\nabla{\cdot}\bp\|^2 - \varepsilon C_{\mathrm{PF}} h_{\Omega} \left(C_{\mathrm{PF}} h_{\Omega}\|\nabla{\cdot}\bp\|^2 + \frac{1}{4 C_{\mathrm{PF}} h_{\Omega}}\|\nabla v\|^2\right)\\
= {} & \frac{\varepsilon}{2}\|\bp\|^2 + \frac{\varepsilon}{4}\|\nabla v\|^2 + \|\nabla{\cdot}\bp\|^2 \big(\LO^2 - \varepsilon C_{\mathrm{PF}}^2 h_{\Omega}^2\big),
\end{align*}
where we have also employed the Green theorem $(\nabla v,\bp) = - (\nabla{\cdot}\bp,v)$ and the Poincar\'e--Friedrichs inequality $\|v\| \leq C_{\mathrm{PF}} h_{\Omega}\|\nabla v\|$ (here $h_{\Omega}$ is the diameter of $\Omega$ and $C_{\mathrm{PF}} \leq 1$ a generic constant). The assertion~\eqref{eq:Coercivity:LSMFE} follows by choosing, e.g., $\varepsilon = \LO^2 / (2 C_{\mathrm{PF}}^2 h_\Omega^2 )$. Note that, employing $C_{\mathrm{PF}} =1$, the constant $C$ in~\eqref{eq:Coercivity:LSMFE} can be taken as $8$.

Let now $\bvh\in \RTN\cap \HdivO$ be such that $\nabla{\cdot}\bvh=\Pihp(\nabla{\cdot}\bsig)$ and $v_{\calT}\in \calP^q(\calT)\cap \Hunz$ be an arbitrary function.
Set $q_{\calT}=v_{\calT}-u_{\mathrm{LS}}$ and $\bp_{\calT}=\bvh-\bsig_{\mathrm{LS}}$. Then using~\eqref{eq:OrthoGonality:LSMFE} and~\eqref{eq:Coercivity:LSMFE}, we find
\begin{align*}
\frac{1}{C} \left(\|\bp_{\calT}\|^2+\|\nabla q_{\calT}\|^2\right) \leq {} & \calA(\bvh-\bsig_{\mathrm{LS}},v_{\calT}-u_{\mathrm{LS}};\bp_{\calT},q_{\calT})\\
= {} & \calA(\bvh-\bsig,v_{\calT}-u;\bp_{\calT},q_{\calT})\\
= {} & (\bvh-\bsig+\nabla (v_{\calT}-u),\nabla q_{\calT})+\LO^2 (\nabla{\cdot}(\bvh-\bsig),\nabla{\cdot}\bp_{\calT})\\
{} & +(\bvh-\bsig+\nabla (v_{\calT}-u),\bp_{\calT}).
\end{align*}
Since $\nabla{\cdot}\bvh=\Pihp(\nabla{\cdot}\bsig)$ and $\nabla{\cdot}\bp_{\calT}\in \calP^p(\calT)$, we have $(\nabla{\cdot}(\bvh-\bsig),\nabla{\cdot}\bp_{\calT})=0$. Using the Cauchy--Schwarz and the Young inequality, we then obtain, with the constant $C$ from~\eqref{eq:Coercivity:LSMFE},
\begin{align*}
\|\bp_{\calT}\|+\|\nabla q_{\calT}\| &\leq  2C \left(\|\bsig-\bvh\|+\|\nabla(u-v_{\calT})\|\right),
\end{align*}
which proves the claim owing to the triangle inequality and since $\bvh$ and $v_{\calT}$ are arbitrary.
\end{proof}

The two terms in the error bound from Lemma~\ref{lem:a_pr_LSM} are uncoupled. For the first one, we can again straightforwardly use Theorems~\ref{thm:loc_glob} and~\ref{thm:appr_est}. For the second one, the result of \ifIMA\else Veeser~\fi\cite{Veeser_approx_grads_16} yields
\[
    \min_{v_{\calT}\in \calP^q(\calT)\cap \Hunz}\|\nabla(u-v_{\calT})\|^2 \leq C
    \sum_{K\in \calT} \min_{q_{K} \in \calP^q(K)}\|\nabla (u-q_{K})\|_K^2,
\]
where the constant $C$ depends only on the space dimension $\dim$, the shape-regularity parameter $\kappa_\calT$ of $\calT$, and the polynomial degree $q$,
which is again optimal.

Finally, a localized estimate for the error $\nabla{\cdot}(\bsig-\bsig_{\mathrm{LS}})$ follows by the combination of the above results with the following lemma.

\begin{lemma}[A priori bound on the divergence for least-squares mixed finite element methods] \label{lem:a_pr_LSM_div} Let $(\bsig_{\mathrm{LS}},u_{\mathrm{LS}})$ be the least-squares mixed finite solution solving~\eqref{eq:DLSMF}, approximating $(\bsig,u)$ from~\eqref{eq:LSMF}. Then
\begin{align*}
\LO^2 \|\nabla{\cdot}(\bsig-\bsig_{\mathrm{LS}})\|^2 \leq {} & \LO^2 \|\nabla{\cdot}\bsig-\Pihp(\nabla{\cdot}\bsig)\|^2+ \|\nabla(u-u_{\mathrm{LS}})\|^2 \\
{} & +\min_{\substack{\bvh\in \RTN\cap \HdivO\\\nabla{\cdot}\bvh=\Pihp(\nabla{\cdot}\bsig)}} \|\bsig-\bvh\|^2.
\end{align*}
\end{lemma}

\begin{proof}
Again let $\bsh\in \RTN\cap \HdivO$ be such that $\nabla{\cdot}\bsh=\Pihp(\nabla{\cdot}\bsig)$. Using~\eqref{eq:LSMF2} and~\eqref{eq:DLSMF2}, we have
\begin{align*}
\LO^2 \|\nabla{\cdot}(\bsig-\bsig_{\mathrm{LS}})\|^2
= {} & \LO^2 (\nabla{\cdot}(\bsig-\bsig_{\mathrm{LS}}),\nabla{\cdot}(\bsig-\bsig_{\mathrm{LS}}))\\
= {} & \LO^2 (\nabla{\cdot}(\bsig-\bsig_{\mathrm{LS}}),\nabla{\cdot}(\bsig-\bsh))
+ \LO^2 (\nabla{\cdot}(\bsig-\bsig_{\mathrm{LS}}),\nabla{\cdot}(\bsh-\bsig_{\mathrm{LS}}))\\
= {} & \LO^2 (\nabla{\cdot}(\bsig-\bsig_{\mathrm{LS}}),\nabla{\cdot}(\bsig-\bsh))
+ \LO^2 (\nabla{\cdot}(\bsig-\bsig_{\mathrm{LS}}),\nabla{\cdot}(\bsh-\bsig_{\mathrm{LS}}))\\
{} &  - \calA(\bsig-\bsig_{\mathrm{LS}},u-u_{\mathrm{LS}};\bsh-\bsig_{\mathrm{LS}},0)\\
= {} & \LO^2 (\nabla{\cdot}(\bsig-\bsig_{\mathrm{LS}}),\nabla{\cdot}(\bsig-\bsh))-((\bsig-\bsig_{\mathrm{LS}})+\nabla(u-u_{\mathrm{LS}}),\bsh-\bsig_{\mathrm{LS}})\\
= {} & \LO^2 (\nabla{\cdot}(\bsig-\bsh),\nabla{\cdot}(\bsig-\bsh))-((\bsig-\bsig_{\mathrm{LS}})+\nabla(u-u_{\mathrm{LS}}),\bsh-\bsig_{\mathrm{LS}}),
\end{align*}
where we used that $(\nabla{\cdot}\bsig_{\mathrm{LS}},\nabla{\cdot}(\bsig-\bsh))=(\nabla{\cdot}\bsh,\nabla{\cdot}(\bsig-\bsh))=0$ since both $\nabla{\cdot}\bsig_{\mathrm{LS}}$ and $\nabla{\cdot}\bsh$ belong to $\Pp(\calT)$ in the last equality. Adding and subtracting $\bsig$ in the second term on the right-hand side above and applying the Cauchy--Schwarz and Young inequalities implies that
\begin{align*}
& \LO^2 \|\nabla{\cdot}(\bsig-\bsig_{\mathrm{LS}})\|^2+\|\bsig-\bsig_{\mathrm{LS}}\|^2 \\
= {} & \LO^2 \|\nabla{\cdot}\bsig-\Pihp(\nabla{\cdot}\bsig)\|^2-(\nabla(u-u_{\mathrm{LS}}),\bsig-\bsig_{\mathrm{LS}})\\
{} &
-(\nabla(u-u_{\mathrm{LS}}),\bsh-\bsig)-(\bsig-\bsig_{\mathrm{LS}},\bsh-\bsig)\\
\leq {} & \LO^2 \|\nabla{\cdot}\bsig-\Pihp(\nabla{\cdot}\bsig)\|^2+ \|\nabla(u-u_{\mathrm{LS}})\|^2+ \|\bsig-\bsh\|^2+ \|\bsig-\bsig_{\mathrm{LS}}\|^2.
\end{align*}
We infer that
\begin{align*}
\LO^2 \|\nabla{\cdot}(\bsig-\bsig_{\mathrm{LS}})\|^2&\leq\LO^2 \|\nabla{\cdot}\bsig-\Pihp(\nabla{\cdot}\bsig)\|^2 + \|\nabla(u-u_{\mathrm{LS}})\|^2 + \|\bsig-\bsh\|^2.
\end{align*}
This finishes the proof since $\bsh$ is arbitrary.
\end{proof}

\appendix

\section{$p$-robust constrained--unconstrained equivalence on a simplex} \label{app:constr}

We present in this appendix a way to remove the divergence constraint on a single simplex, and we do this in a polynomial-degree-robust way.
This equivalence of constrained and unconstrained local-best approximations is an important consequence of the result of \ifIMA\else Costabel and McIntosh~\fi\cite[Corollary~3.4]{Cost_McInt_Bog_Poinc_10}.

Recall the notation $\Eloc(\bm{v})$ from~\eqref{eq:local_minimizers}, where $\RTNK = \Pp(K;\R^\dim) + \bx \Pp(K)$ is the Raviart--Thomas--N\'ed\'elec space of degree $p$ on the simplex $K$, as well as that $h_K$ denotes the diameter of $K$ and $\varrho_K$ the diameter of the largest ball inscribed in $K$.

\begin{lemma}[Local $p$-robust constrained--unconstrained equivalence]\label{lem:local_const_equivalence}
Let a simplex $K \subset \R^d$, $d \geq 1$, and $\bv \in \HdivK$ be fixed. Let $\bth$ be defined as in~\eqref{eq_L2_proj_constraint}. Then, there exists a constant $C$, depending only on the space dimension $\dim$ and the shape-regularity parameter $\kappa_K \eq h_K / \varrho_K$ of $K$, such that
\begin{align}\label{eq:local_const_equivalence}
\Eloc(\bv) \leq \norm{\bv-\bth}_K+ \frac{h_K}{p+1}\norm{\nabla{\cdot}(\bv - \bth)}_K
\leq C \Eloc(\bv).
\end{align}
\end{lemma}
\begin{proof}
Since $\nabla{\cdot}\bth = \Pihp(\nabla{\cdot} \bv)$ from~\eqref{eq_L2_proj_constraint}, the first inequality in~\eqref{eq:local_const_equivalence} is obvious, so we show the second one. Therein, $\frac{h_K}{p+1}\norm{\nabla{\cdot}(\bv-\bth)}_K \leq \Eloc(\bv)$ trivially holds true for the same reason, so it remains only to bound~$\norm{\bv-\bth}_K$.

Let $\wbth$ be the elementwise $\bL^2$-projection of $\bv$ into $\RTN$, so that
\[
[\Eloc(\bv)]^2 = \norm{\bv-\wbth}_K^2 + \frac{h_K^2}{(p+1)^2}\norm{\nabla{\cdot}\bv - \Pihp(\nabla{\cdot} \bv)}_K^2.
\]
It follows from~\cite[Corollary~3.4]{Cost_McInt_Bog_Poinc_10} that there exists
$\bvK \in \RTNK$ such that $\nabla{\cdot}\bvK = \Pihp(\nabla{\cdot}\bv)$ and
\begin{equation}\label{eq:local_stab}
\norm{\bvK-\wbth}_K \leq C \sup_{\substack{\varphi\in H^1_0(K)\\ \norm{\nabla\varphi}_K=1}}\left\{ (\Pihp(\nabla{\cdot}\bv) - \nabla{\cdot} \wbth,\varphi)_K\right\},
\end{equation}
where $C$ only depends on $\dim$ and $\kappa_{K}$. Since $(\nabla{\cdot}\bv,\varphi)_K+(\bv,\nabla\varphi)_K =0$, and since also $(\nabla{\cdot}\wbth,\varphi)_K+(\wbth,\nabla\varphi)_K =0$ for all $\varphi \in H^1_0(K)$, we see that
\[
(\Pihp(\nabla{\cdot}\bv),\varphi)_K - (\nabla{\cdot} \wbth,\varphi)_K = (\Pihp(\nabla{\cdot}\bv)-\nabla{\cdot}\bv,\varphi-\Pihp (\varphi))_K - (\bv-\wbth,\nabla\varphi)_K,
\]
where we have also freely subtracted $\Pihp (\varphi)$.
Therefore, the inequality~\eqref{eq:local_stab} combined with the approximation bound $\norm{\varphi-\Pihp(\varphi)}_K \leq C \frac{h_K}{p+1}\norm{\nabla\varphi}_K$, with a constant $C$ depending only on $\dim$ and $\kappa_{K}$, implies that
\[
\norm{\bvK-\wbth}_K \leq C \left\{ \norm{\bv-\wbth}^2_K + \Big[\frac{h_K}{p+1}\norm{\nabla{\cdot}\bv - \Pihp(\nabla{\cdot}\bv)}_K\Big]^2\right\}^\frac{1}{2} = C \Eloc(\bv).
\]

Finally, owing to the triangle inequality $\norm{\bv-\bvK}_K \leq \norm{\bv-\wbth}_K+\norm{\wbth-\bvK}_K$, we infer that $\norm{\bv-\bvK}_K \leq C \Eloc(\bv)$.
Consequently, the definition of $\bth$ as the minimizer in~\eqref{eq_L2_proj_constraint} implies that $\norm{\bv-\bth}_K \leq \norm{\bv-\bvK}_K$, and this yields the second bound in~\eqref{eq:local_const_equivalence}.
\end{proof}

\ifIMA
\else
\bibliographystyle{acm}
\fi
\bibliography{biblio}
\end{document}